\documentclass[a4paper,reqno,11pt]{amsart}

\usepackage{geometry}
\usepackage[T1]{fontenc}
\usepackage[utf8]{inputenc}
\usepackage{xspace,
	amsfonts}
\usepackage{amsmath}
\usepackage{amssymb}
\usepackage{graphicx,
	bbm,
	tikz,
	caption,
	subcaption,
	}
\usepackage{xfrac}
\usepackage{dsfont}
\usepackage{amsthm}
\usepackage{mathtools,mathrsfs}
\usepackage{enumitem}
\usepackage{verbatim}
\usepackage[autostyle]{csquotes}

\usepackage{cite}

\mathtoolsset{showonlyrefs}

\usepackage[colorlinks=true,urlcolor=blue, citecolor=red,linkcolor=blue,
linktocpage,pdfpagelabels,bookmarksnumbered,bookmarksopen]{hyperref}

\usepackage{xcolor}

\usepackage{tikz}
\usetikzlibrary{shapes.geometric,calc,decorations.markings,math}

\tikzstyle{nodo}=[circle,draw,fill,inner sep=1.5pt]
\tikzstyle{infinito}=[circle,inner sep=0pt,minimum size=0mm]
\tikzset{midarrow/.style={decoration={markings,
		mark=at position 0.5 with {\arrow{stealth}} 
	},
	postaction={decorate}}}

\tikzset{midarrowtwo/.style={decoration={markings,
			mark=at position 0.5 with {\arrow{stealth reversed}} 
		},
		postaction={decorate}}}

\tikzset{every loop/.style={min distance=10mm,in=300,out=240,looseness=10}}

\tikzset{place/.style={circle,thick,draw=blue!75,fill=blue!20,minimum
		size=6mm}}
\tikzset{place2/.style={circle,thick,draw=red!75,fill=red!20,minimum
		size=6mm}}
	
\tikzset{every picture/.style={transform shape}}
	
\theoremstyle{plain}
\newtheorem{theorem}{Theorem}[section]
\newtheorem{proposition}[theorem]{Proposition}

\newtheorem{lemma}[theorem]{Lemma}
\theoremstyle{definition}
\newtheorem{definition}[theorem]{Definition}
\newtheorem{remark} [theorem] {Remark}

\def \N {\mathbb{N}}
\def \R {\mathbb{R}}
\def \Z {\mathbb{Z}}

\def \G {\mathcal{G}}
\def \Gm {\G_{\mathrm{mod}}}
\def \E {\mathcal{E}}
\def \HH {\mathcal{H}}

\def \S {\mathcal{S}}
\def \M {\mathcal{M}}
\def \K {\mathcal{K}}

\def \dd {\text{d}}
\def \Q {\mathcal{Q}}
\def \Qm {\Q_{\mathrm{mod}}}
\def \Qo {\Q_{\mathrm{open}}}

\def \L {\mathcal{L}}
\def \D {\mathcal{D}}
\def \NN {\mathcal{N}}
\def \cell {\mathcal{C}}
\def \ep {\varepsilon}
\def \v {\mathbf{v}}
\def \w {\mathbf{w}}

\DeclareMathOperator\supp{supp}

\newcommand{\dx}{\,dx}
\newcommand{\red}[1]{\color{red}{#1}}

\title[NLS ground states on defected grids with unbounded defects]{Normalized nonlinear Schr\"odinger ground states on defected grids with unbounded defects}

\author[W. Borrelli]{William Borrelli}
\address{W. Borrelli: Sapienza Universit\`a di Roma, Dipartimento di Matematica, Piazzale A. Moro 2, I-00185 Roma, Italy}
\email{william.borrelli@uniroma1.it}

\author[N. Cangiotti]{Nicol\`o Cangiotti}
\address{ N.\,Cangiotti: SIS Paolino d'Aquileia, 
Via Istituto Tecnico Agrario, 42,  I-33043 Cividale del Friuli, Udine, Italy}
\email{nicolocangiotti@paolinodaquileia.edu.it}

\author[S. Dovetta]{Simone Dovetta}
\address{S. Dovetta: Politecnico di Torino, Dipartimento di Scienze Matematiche ``G.L. Lagrange'', Corso Duca degli Abruzzi, 24, 10129, Torino, Italy}
\email{simone.dovetta@polito.it}

\author[L. Tentarelli]{Lorenzo Tentarelli}
\address{L. Tentarelli: Politecnico di Torino, Dipartimento di Scienze Matematiche ``G.L. Lagrange'', Corso Duca degli Abruzzi, 24, 10129, Torino, Italy}
\email{lorenzo.tentarelli@polito.it}

\date{\today}



\begin{document}

\maketitle

\begin{abstract}
We study normalized ground states for the nonlinear Schr\"odinger energy on two-dimensional square grids with unbounded defects. We first show that unbounded defects do not necessarily destroy the dimensional crossover, proving its persistence on the half-grid and the quarter-grid. We then consider a prototypical defected grid where a macroscopically two-dimensional region coexists with an unbounded one-dimensional channel. On this graph, normalized ground states do not exist for sufficiently small or sufficiently large masses, showing that the one-dimensional component can dominate the variational problem and induce loss of compactness. Finally, we show that this mechanism is sensitive to the local geometry of the graph: suitable additional defects preserve small-mass nonexistence but restore the existence of ground states for large masses. Our analysis combines small-mass energy asymptotics with a large-mass blow-up argument.
\end{abstract}


\section{Introduction and main results}

Metric graphs, i.e. domains obtained by gluing together several intervals through the identification of some of their endpoints,  provide a natural framework for differential equations on networks, combining a discrete topology with a one-dimensional microscale along the edges. Over the decades, a broad class of evolution and spectral problems has been investigated on these structures. In view of a wide variety of possible applications, specific attention has been devoted to nonlinear Schr\"odinger equations and, in particular, to stationary problems of the form
\begin{equation}
\label{eq:nlse}
u''+|u|^{p-2}u=\lambda u
\qquad\text{on every edge},
\end{equation}
coupled with suitable conditions at the vertices. To date, perhaps the most widely studied vertex conditions are the homogeneous Kirchhoff ones, namely
\begin{equation}
\label{eq:kirchhoff}
\sum_{e\succ \v}\frac{du}{dx_e}(\v)=0
\qquad\text{for every vertex }\v,
\end{equation}
where the sum is taken over all edges incident to the vertex $\v$, and $\displaystyle du/dx_e(\v)$ denotes the outgoing derivative of $u$ at $\v$ along the edge $e$.

This stationary problem has a natural variational formulation. Indeed, given a connected metric graph $\M=(\mathbf{V}_\M, \mathbf{E}_\M)$, for every $p>2$, solutions of \eqref{eq:nlse}--\eqref{eq:kirchhoff} on $\M$ can be found as critical points of the NLS energy $E:H^1(\M)\to\R$
\begin{equation}
	\label{eq:E}
	E(u,\M):=\frac{1}{2}\int_\M |u'|^2\dx-\frac{1}{p}\int_\M \vert u\vert^p\dx
\end{equation}
constrained to the set of functions with prescribed $L^2$-norm (the mass), i.e.
\begin{equation}
	\label{eq-Hmu}
	H^1_\mu(\M)=\left\{u\in H^1(\M)\, :\, \Vert u\Vert^2_{L^2(\M)}=\mu\right\}.
\end{equation}
As a consequence, the parameter $\lambda\in\R$ in \eqref{eq:nlse} arises as a Lagrange multiplier associated with the mass constraint.

Among critical points, a major role is played by ground states with mass $\mu$ (also called normalized ground states), defined as global minimizers of $E$ in $H_\mu^1(\M)$, i.e. $u\in H_\mu^1(\M)$ such that
\begin{equation}
	\label{eq:gsG}
	E(u,\M)=\E_\M(\mu):=\inf_{v\in H_\mu(\M)}E(v,\M).
\end{equation}
The study of these variational problems on metric graphs has attracted considerable interest within the community, also motivated by the fact that their behavior has proved to be particularly sensitive to the specific features of the underlying graph, giving rise to phenomena with no direct counterpart when the same problems are posed on $\R^N$. The literature on the subject is by now very large, and we refer for instance to \cite{KNP} for a comprehensive overview and a thorough list of references. A large body of results concerns noncompact metric graphs with finitely many edges, see e.g. \cite{ACFN-AIHP-2014,ACFN-JDE-2016, ASTcmp, Adami15, AST2, ACT2, ACT1, BMP, BC, BD, BCJS, CGJT, CJS, D18, JS, KMPX, LLZ, NP, PS, PSV, Serra16, Tentarelli16}. 
More recently, increasing attention has been devoted to graphs with infinitely many edges, and in particular to periodic metric graphs. Periodic graphs are specifically interesting because they have the peculiar property of combining the usual one-dimensional microscale of metric graphs with a macroscale whose dimension can be higher, depending on the degree of periodicity of the graph itself. This coexistence of scales with different dimensions is reflected in a new phenomenon called dimensional crossover, originally identified in \cite{Adami19b} for the two-dimensional square grid $\Gamma$ (Figure \ref{fig:grid}), and then extended to more general periodic settings in \cite{Adami18,Adami19a,SV} (see also \cite{Barbera25,Boni25,Dovetta24} for related results). 

Since in the present paper we restrict to the two-dimensional setting and, more precisely, to metric graphs related to the two-dimensional square grid with unit-length edges $\Gamma$, let us describe in detail the dimensional crossover as it takes place on $\Gamma$. This phenomenon can be seen at two different levels. 

\begin{figure}
	\begin{center}
		\begin{tikzpicture}[xscale= 0.5,yscale=0.5]
			\draw[step=2,thin] (0,0) grid (8,8);
			\foreach \x in {0,2,...,8} \foreach \y in {0,2,...,8} \node at (\x,\y) [nodo] {};
			\foreach \x in {0,2,...,8}
			{\draw[dashed,thin] (\x,8.2)--(\x,9.2) (\x,-0.2)--(\x,-1.2) (-1.2,\x)--(-0.2,\x)  (8.2,\x)--(9.2,\x); }
		\end{tikzpicture}
	\end{center}
	\caption{The two-dimensional square grid  $\Gamma$.}
	 \label{fig:grid}
\end{figure}
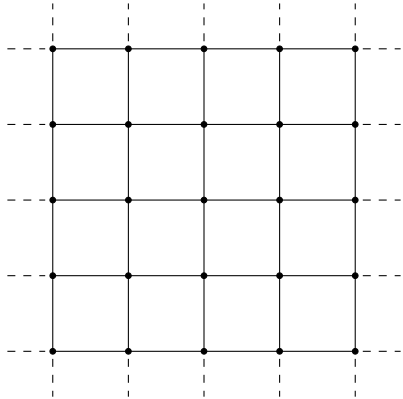

At the level of functional inequalities, $\Gamma$ supports both the one-dimensional Sobolev inequality
\begin{equation}
		\label{eq-sob1d}
\|u\|_{L^\infty(\Gamma)}
\lesssim_\Gamma
\|u'\|_{L^1(\Gamma)}
\qquad
\forall u\in W^{1,1}(\Gamma),
\end{equation}
which holds on every connected noncompact metric graph, and the two-dimensional one (see \cite[Theorem 2.2]{Adami19b})
\begin{equation}
	\label{eq-sob2d}
\|u\|_{L^2(\Gamma)}
\lesssim_\Gamma
\|u'\|_{L^1(\Gamma)}
\qquad
\forall u\in W^{1,1}(\Gamma)\,,
\end{equation}
which is peculiar to $\Z^2$-periodic graphs. 
As a direct consequence, $\Gamma$ supports both the one- and two-dimensional Gagliardo-Nirenberg inequalities
\begin{equation}
\label{eq-gn1d}
\|u\|_{L^p(\Gamma)}^p
\lesssim_{p,\Gamma}
\|u\|_{L^2(\Gamma)}^{\frac p2+1}
\|u'\|_{L^2(\Gamma)}^{\frac p2-1}\qquad\forall u\in H^1(\Gamma),
\end{equation}
\begin{equation}
\label{eq-gn2d}
\|u\|_{L^p(\Gamma)}^p
\lesssim_{p,\Gamma}
\|u\|_{L^2(\Gamma)}^2
\|u'\|_{L^2(\Gamma)}^{p-2}\qquad\forall u\in H^1(\Gamma),
\end{equation}
for every $p>2$. Inequalities \eqref{eq-sob1d} and \eqref{eq-gn1d} are rooted in the one-dimensional microscale of the graph, while \eqref{eq-sob2d} and \eqref{eq-gn2d} are determined by the two-dimensional macroscale. Their simultaneous validity yields the further family of interpolated inequalities for every $p\in[4,6]$
\begin{equation}
\label{eq-gnint}
\|u\|_{L^p(\Gamma)}^p
\lesssim_{p,\Gamma}
\|u\|_{L^2(\Gamma)}^{p-2}
\|u'\|_{L^2(\Gamma)}^2\qquad\forall u\in H^1(\Gamma)\,.
\end{equation}
In the one-dimensional inequality \eqref{eq-gn1d} the $L^p$ norm can be controlled by $\|u'\|_{L^2(\Gamma)}^2$ on the right-hand side only if $p=6$, while in the two-dimensional analogue \eqref{eq-gn2d} only with $p=4$. The interpolated inequalities \eqref{eq-gnint} show that on $\Gamma$ this can actually be done for every intermediate norm with $p\in[4,6]$. 

At the variational level, a second manifestation of the dimensional crossover can be seen in the behavior of $\E_\Gamma(\mu)$, since: \begin{enumerate}[label=(\roman*)]
	\item if $p\in(2,4)$, then $\E_\Gamma(\mu)$ is negative and finite and ground states do exist for every $\mu>0$;
	\item if $p\in[4,6]$, there exists a critical mass $\mu_{p,\Gamma}>0$ such that
	\[
	\E_\Gamma(\mu)\begin{cases}
		\displaystyle =0 & \text{if }\:\mu\leq\mu_{p,\Gamma}\\
		\displaystyle <0 & \text{if }\:\mu>\mu_{p,\Gamma}\,,
	\end{cases}
	\]
	and
	\begin{itemize}
		\item when $p\in(4,6)$, ground states exist if and only if $\mu\geq\mu_{p,\Gamma}$;
		\item when $p=4$, ground states exist if $\mu>\mu_{4,\Gamma}$ and do not exist if $\mu<\mu_{p,\Gamma}$;
		\item when $p=6$, $\E_\Gamma(\mu)=-\infty$ for every $\mu>\mu_{6,\Gamma}$ and ground states do not exist for any value of $\mu$.
	\end{itemize}
\end{enumerate}
In other words, the minimization problem exhibits a purely $L^2$-subcritical behavior only for $p\in(2,4)$, whereas for $p\in[4,6)$ it exhibits the typical $L^2$-critical feature of a threshold mass discriminating between existence and nonexistence. Hence, on the grid $\Gamma$ every power between the two-dimensional $L^2$-critical one $p=4$ and the one-dimensional one $p=6$ is somehow critical.

Since the dimensional crossover is rooted in the coexistence of different dimensional scales and, in particular, in the two-dimensional macroscale generated by the periodicity of $\Gamma$, it is natural to ask how the phenomenon may be affected by topological perturbations of the grid possibly breaking the periodic structure. This question was first addressed in \cite{Dovetta22}, where defected grids are obtained from $\Gamma$ by removing connected clusters of edges called defects. Each defect can be either bounded (i.e. containing finitely many edges) or unbounded (i.e. with infinitely many edges).  

 For the persistence of the two-dimensional Sobolev inequality \eqref{eq-sob2d} on such defected grids, a thorough analysis is developed  in \cite{Dovetta22}. While it is proved that uniformly bounded defects always preserve \eqref{eq-sob2d}, this is not the case when turning to unbounded defects, as they can change the large-scale geometry of the graph and create essentially one-dimensional regions. For unbounded defects, sufficient topological conditions for \eqref{eq-sob2d} to persist are given in \cite{Dovetta22}, and a rather broad description of the robustness of the dimensional crossover at the level of Sobolev inequalities is provided.

The picture for normalized ground states is less understood. In this direction, the results of \cite{Dovetta22} cover only bounded defects, and thus refer only to defected grids on which the dimensional crossover for Sobolev inequalities is preserved. The case of unbounded defects, on the contrary, is untouched. 

\smallskip
The aim of the present paper is to advance this line of research, focusing on the persistence and failure of the dimensional crossover at the level of NLS ground states on grids with unbounded defects. In particular, the main question we want to address is the following. Suppose that an unbounded defect breaks the two-dimensional macrostructure of the grid, leaving a region with a genuinely two-dimensional macroscale and another one with a one-dimensional behavior. Which of the two structures determines the ground state problem?

This question is not answered by the failure of the two-dimensional Sobolev inequality alone. Indeed, even though this shows that the functional framework underlying the dimensional crossover is lost, it does not directly imply nonexistence of ground states. The variational problem involves a further competition: minimizing sequences may either exploit the part of the graph where the two subregions are glued together, and converge to an actual ground state, or lose compactness and drift away along the one-dimensional region. Since we already know that one-dimensional subgraphs make the two-dimensional Sobolev inequality fail, the latter alternative would confirm the disruptive role of these structures also for the dimensional crossover at the variational level.

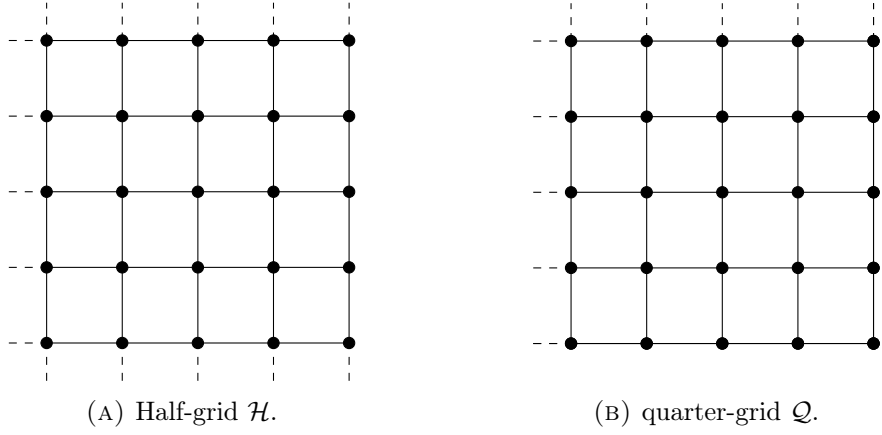
\begin{figure}[htbp]
	\centering
	\begin{subfigure}[b]{0.48\textwidth}
		\centering
		\begin{tikzpicture}[scale=1,>=latex]
			\draw[step=1] (0,0) grid (4,4);
			\foreach \i in {0,...,4}{
				\draw[dashed] (\i,4)--(\i,4.5);}
			\foreach \j in {0,...,4}{
				\draw[dashed] (-.5,\j)--(0,\j);}
			\foreach \i in {0,...,4}{
				\draw[dashed] (\i,-.5)--(\i,0);}
			\foreach \i in {0,...,4}
			\foreach \j in {0,...,4}
			{\node at (\i,\j) [nodo] {};}
		\end{tikzpicture}
		\caption{Half-grid $\HH$.}
		\label{fig:halfgrid_a}
	\end{subfigure}
	\hspace{-1cm}
	\begin{subfigure}[b]{0.48\textwidth}
		\centering
		\begin{tikzpicture}[scale=1,>=latex]
			\draw[step=1] (0,0) grid (4,4);
			\foreach \j in {0,...,4}{
				\draw[dashed] (-.5,\j)--(0,\j);}
			\foreach \i in {0,...,4}{
				\draw[dashed] (\i,4)--(\i,4.5);}
			\foreach \i in {0,...,4}
			\foreach \j in {0,...,4}
			{\node at (\i,\j) [nodo] {};}
			\draw (0,0)--(4,0);
			\draw (4,0)--(4,4);
			\draw[dashed] (-.5,0)--(0,0);
			\draw[dashed] (4,4)--(4,4.5);
			\foreach \i in {0,...,4}{
				\node at (\i,0) [nodo] {};
				\node at (4,\i) [nodo] {};
				\node at (\i,-.5) [infinito] {};}
		\end{tikzpicture}
		\caption{quarter-grid $\Q$.}
		\label{fig:quartgrid_b}
	\end{subfigure}
	\caption{The half-grid and the quarter-grid.}
	\label{fig:halfquarter}
\end{figure}

Before addressing this question, we first show that unbounded defects do not by themselves destroy the dimensional crossover. Consider the half-grid $\HH$ and the quarter-grid $\Q$ depicted in Figure \ref{fig:halfquarter}. Clearly, both graphs can be seen as defected grids obtained from $\Gamma$ by removing a unique unbounded defect. Nonetheless, they both retain a two-dimensional macroscale. Hence, the two-dimensional Sobolev inequality continues to hold (see \cite[Theorem 2.25]{Dovetta22}). Moreover, the dimensional crossover for ground states persists, as proved in the next theorem.
\begin{theorem}[Dimensional crossover on the half-grid and the quarter-grid]\label{thm:half-quart}
	Let $X$ be either the half-grid $\HH$ (Figure {\sc \ref{fig:halfgrid_a}}) or the quarter-grid $Q$ (Figure {\sc \ref{fig:quartgrid_b}}). The following statements hold:
	\begin{enumerate}[label=\arabic*)]
		\item if $p\in(2,4)$, then $\E_X(\mu)$ is negative and finite and there exists a ground state with mass $\mu$ for every $\mu>0$;
		\item if $p\in[4,6)$, then there exists a critical mass $\mu_{p,X}>0$ (depending only on $p$ and $X$) such that
		\begin{equation}
			\label{eq-energylevelHQ}
			\E_{X}(\mu)\begin{cases}
				=0 & \text{if }\mu\leq\mu_{p,X}\\
				<0 & \text{if }\mu>\mu_{p,X}\,,
			\end{cases}
		\end{equation}
		and
		\begin{enumerate}[label=\alph*)]
			\item if $p\in(4,6)$, then ground states with mass $\mu$ exist if and only if $\mu\geq\mu_{p,X}$;
			\item if $p=4$, then ground states with mass $\mu$ exist when $\mu>\mu_{4,X}$, while they do not exist when $\mu<\mu_{4,X}$.
		\end{enumerate}
	\end{enumerate}
	Furthermore,
	\begin{gather}
		\label{eq-massrelationHQp}
		\mu_{p,\Q}<\mu_{p,\HH}<\mu_{p,\Gamma}\quad \text{for}\quad p\in(4,6),\\[.2cm]
		\label{eq-massrelationHQ4}
		\mu_{4,\Q}\leq\mu_{4,\HH}\leq\mu_{4,\Gamma}\,,
	\end{gather}
	and
	\begin{gather}
		\label{eq-energyrelationH}
		\E_{\HH}(\mu)\begin{cases}
			<\E_{\Gamma}(\mu) & \text{if }p\in(2,4)\text{ or }p\in[4,6)\text{ and }\mu>\mu_{p,\HH}\\
			=\E_{\Gamma}(\mu) & \text{if }p\in[4,6)\text{ and }\mu\leq\mu_{p,\HH}\,,
		\end{cases}
		\\[.2cm]
		\label{eq-energyrelationQ}
		\E_{\Q}(\mu)\begin{cases}
			<\E_{\HH}(\mu) & \text{if }p\in(2,4)\text{ or }p\in[4,6)\text{ and }\mu>\mu_{p,\Q}\\
			=\E_{\HH}(\mu) & \text{if }p\in[4,6)\text{ and }\mu\leq\mu_{p,\Q}\,.
		\end{cases}
	\end{gather}
\end{theorem}

We then turn to grids in which a two-dimensional region and a one-dimensional one compete. Since it is clearly impossible to consider the most general structure with such a property, we develop our analysis on the prototypical defected grid $\G$ depicted in Figure \ref{fig:half+lad_1}. This graph is obtained from the grid $\Gamma$ by the removal of two unbounded defects, and can also be interpreted as the union of a half-grid $\HH$ as in Figure {\sc\ref{fig:halfgrid_a}} with a $\Z$-periodic ladder graph $\L_1$ as in Figure \ref{fig:ladder1}. Heuristically, the left part of $\G$ retains a two-dimensional macrostructure, whereas the half-ladder on the right provides an unbounded macroscopically one-dimensional channel. By the results in \cite{Dovetta22}, it is immediate to verify that the two-dimensional Sobolev inequality does not hold on $\G$, and the dimensional crossover at that level is lost. As for ground states, our main results are the following.

\begin{figure}[t]
	\centering
	\begin{tikzpicture}[scale=1,>=latex]
		\draw[step=1] (-2,-1) grid (3,4);
		\foreach \i in {-2,...,3}{
			\draw[dashed] (\i,4)--(\i,4.5);}
		\foreach \j in {-1,...,4}{
			\draw[dashed] (-2.5,\j)--(-2,\j);}
		\foreach \i in {-2,...,3}{
			\draw[dashed] (\i,-1.5)--(\i,-1);}
		\foreach \i in {-2,...,3}
		\foreach \j in {-1,...,4}
		{\node at (\i,\j) [nodo] {};}
		\draw (3,1)--(7,1);
		\draw (3,2)--(7,2);
		\draw (4,1)--(4,2);
		\draw (5,1)--(5,2);
		\draw (6,1)--(6,2);
		\draw (7,1)--(7,2);
		\draw[dashed] (7,1)--(7.5,1);
		\draw[dashed] (7,2)--(7.5,2);
		\node at (7,1) [nodo] {};
		\node at (7,2) [nodo] {};
		\node at (6,1) [nodo] {};
		\node at (6,2) [nodo] {};
		\node at (5,1) [nodo] {};
		\node at (5,2) [nodo] {};
		\node at (4,1) [nodo] {};
		\node at (4,2) [nodo] {};
	\end{tikzpicture}
	\caption{The defected grid $\G$.}
	\label{fig:half+lad_1}
\end{figure}
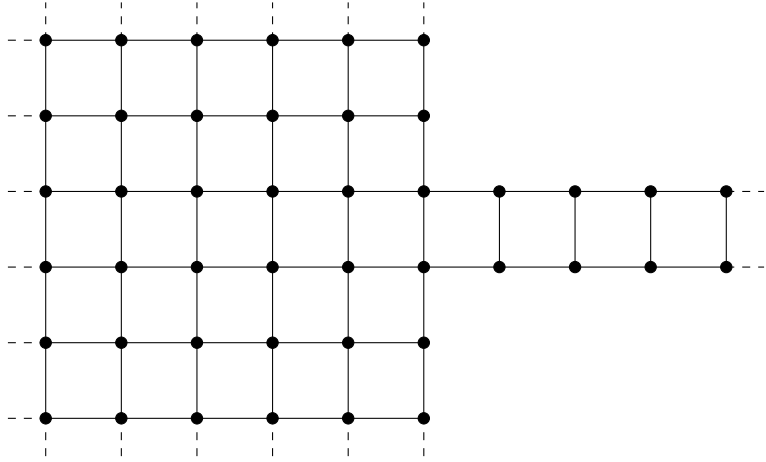

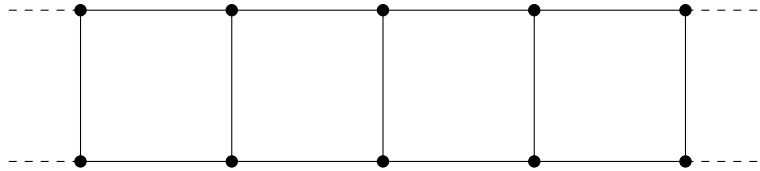
\begin{figure}[!ht]
	\centering
	\begin{tikzpicture}[scale=1,>=latex]
		\draw[step=2] (0,0) grid (8,2);
		\draw[dashed] (8,0)--(9,0);
		\draw[dashed] (8,2)--(9,2);
		\draw[dashed] (0,0)--(-1,0);
		\draw[dashed] (0,2)--(-1,2);
		\foreach \i in {0,...,4}
		\foreach \j in {0,1}
		{\node at (\i*2,\j*2) [nodo] {};}
	\end{tikzpicture}
	\caption{The ladder graph $\L_1$ with unit edges.}
	\label{fig:ladder1}
\end{figure}

\begin{theorem}[Nonexistence for small mass]\label{thm:noexsmall}
	For every $p\in(2,6)$ there exists $\mu_1>0$ such that normalized ground states with mass $\mu$ do not exist on $\G$ (depicted in Figure \ref{fig:half+lad_1}) for any $\mu\in(0,\mu_1)$.
\end{theorem}

\begin{theorem}[Nonexistence for large mass]\label{thm:noexlarge}
	For every $p\in(2,6)$ there exists $\mu_2>0$ such that normalized ground states with mass $\mu$ do not exist on $\G$ (depicted in Figure \ref{fig:half+lad_1}) for any $\mu>\mu_2$.
\end{theorem}
Even though both theorems provide nonexistence of ground states, the two results arise from different mechanisms.

In the small-mass regime covered by Theorem \ref{thm:noexsmall} we actually observe the competition between two-dimensional and one-dimensional behaviors. Indeed, inspired by the recent works \cite{Barbera25,Dovetta24}, we derive asymptotic estimates for the energy level of functions supported on different regions of $\G$ as the mass tends to zero. Roughly speaking, when a function with mass $\mu$ is supported on the left part of $\G$, which is macroscopically two-dimensional, as $\mu\to0$ its energy is at least $-\mathcal O\!\left(\mu^{\frac{2}{4-p}}\right)$ if $p\in(2,4)$, and it is even strictly positive if $p\in[4,6)$. On the contrary, for every $p\in(2,6)$, when $\mu$ is sufficiently small there exist functions with mass $\mu$ supported on the right part of $\G$ attaining energy level equal to $\displaystyle -C_p\mu^{\frac{p+2}{6-p}}$. This shows that, when the mass is small, lowest-energy functions exploiting only the half-ladder on the right of $\G$ display a purely one-dimensional behavior. Since
\[
\frac{2}{4-p}
>
\frac{p+2}{6-p}
\qquad
\text{for }p\in(2,4),
\]
the ladder is energetically convenient as $\mu\to0$. Hence, minimizing sequences on $\G$ escape at infinity along the one-dimensional part of the graph, while any function carrying a nonzero amount of mass in the two-dimensional region has strictly larger energy.

The large-mass regime is different. In this setting, nonuniqueness of ground states is obtained by showing first that, if any such state exists, when $\mu\to+\infty$ it should become a highly concentrated function on $\G$ having negative energy on a unique edge of the graph. This is achieved with a thorough blow-up analysis (see also the recent paper \cite{CDGJT} for related results). On $\G$, however, every edge is contained in a copy of the ladder $\L_1$. This fact can be exploited to start with the supposed ground state and build a competitor with strictly lower energy supported sufficiently far away on the half-ladder on the right of $\G$, proving Theorem \ref{thm:noexlarge} by contradiction. 

Theorems \ref{thm:noexsmall}--\ref{thm:noexlarge} show that the one-dimensional part of a defected grid can dominate the normalized ground state problem even when a large two-dimensional region is still present. For both small and large masses, the energetically favorable behavior is the one associated with the one-dimensional channel, and minimizing sequences can escape through it. However, the outlines of their proofs highlight a major difference: the one-dimensional region of $\G$ plays a crucial role in both settings, while the two-dimensional one seems to be relevant only in the small-mass regime. 

This is the reason why at present we are not able to deal with intermediate masses. In particular, our results do not yet prove that $\G$ has no normalized ground states for every $\mu>0$. Establishing this would give a complete variational counterpart of the failure of the two-dimensional Sobolev inequality for this geometry. 
More generally, it would be interesting to understand which geometric properties of an unbounded defect ensure that a macroscopically one-dimensional region dominates every competing two-dimensional component. 

In this regard, however, we stress that one must be very careful with the rough intuition that the presence of a macroscopically one-dimensional region always rules out normalized ground states, regardless of any specific property of the large-scale two-dimensional part of the graph. In fact, at this level of generality such a heuristic is false. We prove this in the last result of the paper, by showing that the further removal of a few edges on the left part of $\G$ does not affect the behavior of small-mass ground states (since the two-dimensional macroscale is preserved), but may produce local topological structures energetically favorable for highly peaked functions, thus restoring compactness and existence of ground states with large mass. 

\begin{theorem}
	\label{thm:robust}
	Let $p\in(2,6)$ and $\G$ be as in Figure \ref{fig:half+lad_1}. Thinking of $\G$ as the non-disjoint union of the half-grid $\HH$ (Figure \ref{fig:halfgrid_a}) with a horizontal ladder $\L_1$ (Figure \ref{fig:ladder1}), let $\Gm$ be any defected grid obtained from $\G$ by removing finitely many edges in $\G\setminus\L_1$ in such a way that $\Gm$ has a vertex with degree 1 (see e.g. Figure \ref{fig:half+lad_1_mod}). Then there exist $\mu_2>\mu_1>0$ such that
	\begin{itemize}
		\item[(i)] normalized ground states with mass $\mu$ do not exist on $\Gm$ for any $\mu\in(0,\mu_1)$;
		
		\item[(ii)] normalized ground states with mass $\mu$ exist on $\Gm$ for every $\mu>\mu_2$.
	\end{itemize}
\end{theorem}

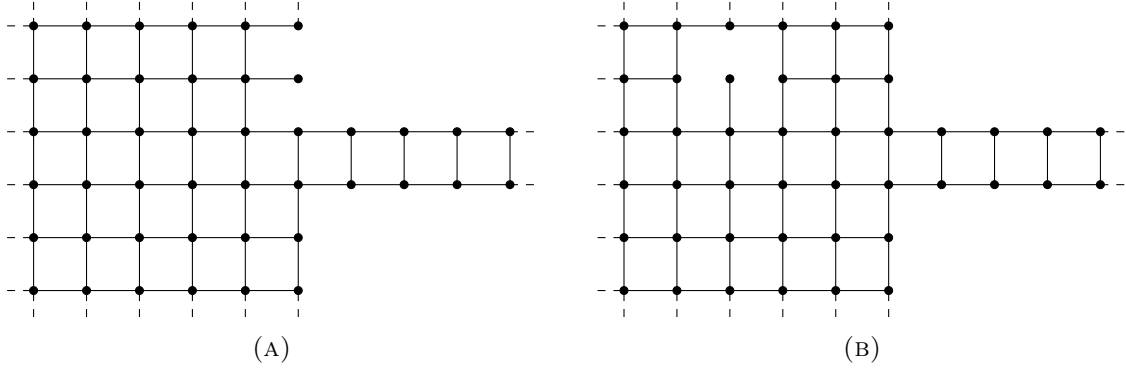
\begin{figure}[t]
\centering
\begin{subfigure}[b]{0.48\textwidth}
	\centering
	\begin{tikzpicture}[scale=0.7,>=latex]
		\draw[step=1] (-2,-1) grid (2,4);
		\foreach \i in {-2,...,3}{
			\draw[dashed] (\i,4)--(\i,4.5);}
		\foreach \j in {-1,...,4}{
			\draw[dashed] (-2.5,\j)--(-2,\j);
			\draw (2,\j)--(3,\j);}
		\foreach \i in {-2,...,3}{
			\draw[dashed] (\i,-1.5)--(\i,-1);}
		\foreach \i in {-2,...,3}
		\foreach \j in {-1,...,4}
		{\node at (\i,\j) [nodo] {};}
		\draw (3,1)--(7,1);
		\draw (3,2)--(7,2);
		\draw (4,1)--(4,2);
		\draw (5,1)--(5,2);
		\draw (6,1)--(6,2);
		\draw (7,1)--(7,2);
		\draw (3,-1)--(3,2);
		\draw[dashed] (7,1)--(7.5,1);
		\draw[dashed] (7,2)--(7.5,2);
		\node at (7,1) [nodo] {};
		\node at (7,2) [nodo] {};
		\node at (6,1) [nodo] {};
		\node at (6,2) [nodo] {};
		\node at (5,1) [nodo] {};
		\node at (5,2) [nodo] {};
		\node at (4,1) [nodo] {};
		\node at (4,2) [nodo] {};
	\end{tikzpicture}
	\caption{ }
\end{subfigure}
\begin{subfigure}[b]{0.48\textwidth}
	\centering
	\begin{tikzpicture}[scale=0.7,>=latex]
		\draw[step=1] (-2,-1) grid (2,4);
		
		\draw[white,line width=2pt] (-1,3)--(1,3);
		
		\draw[white,line width=2pt] (0,3)--(0,4);
		
		\foreach \i in {-2,...,3}{
			\draw[dashed] (\i,4)--(\i,4.5);}
		\foreach \j in {-1,...,4}{
			\draw[dashed] (-2.5,\j)--(-2,\j);
			\draw (2,\j)--(3,\j);}
		\foreach \i in {-2,...,3}{
			\draw[dashed] (\i,-1.5)--(\i,-1);}
		
		\foreach \i in {-2,...,3}
		\foreach \j in {-1,...,4}
		{\node at (\i,\j) [nodo] {};}
		
		\draw (3,-1)--(3,4);
		
		\draw (3,1)--(7,1);
		\draw (3,2)--(7,2);
		\draw (4,1)--(4,2);
		\draw (5,1)--(5,2);
		\draw (6,1)--(6,2);
		\draw (7,1)--(7,2);
		
		\draw[dashed] (7,1)--(7.5,1);
		\draw[dashed] (7,2)--(7.5,2);
		
		\node at (7,1) [nodo] {};
		\node at (7,2) [nodo] {};
		\node at (6,1) [nodo] {};
		\node at (6,2) [nodo] {};
		\node at (5,1) [nodo] {};
		\node at (5,2) [nodo] {};
		\node at (4,1) [nodo] {};
		\node at (4,2) [nodo] {};
	\end{tikzpicture}
	\caption{ }
\end{subfigure}
\caption{Examples of defected grids $\Gm$ as in Theorem \ref{thm:robust}.}
\label{fig:half+lad_1_mod}
\end{figure}


\subsection{Organization of the paper}

The paper is organized as follows:
\begin{itemize}
	\item in Section \ref{sec-preliminaries} we present some preliminary results on the asymptotic behavior of normalized ground states on ladder graphs (Figure \ref{fig:ladder});
	\item in Section \ref{sec-halfquart} we show the existence of cases in which the phenomenon of the dimensional crossover for normalized ground states is preserved, namely we prove Theorem \ref{thm:half-quart};
	\item in Section \ref{sec-nonexistence} we prove the two nonexistence results, namely Theorems \ref{thm:noexsmall} and \ref{thm:noexlarge};
	\item in Section \ref{sec-robust} we prove Theorem \ref{thm:robust}.
\end{itemize}


\section{Preliminary results on ladder graphs}
\label{sec-preliminaries}

In this section we establish a first result on the ground state energy level on ladder graphs. 
A ladder graph is a $\Z$-periodic graph obtained by gluing together infinitely many copies of a compact graph formed by three edges of the same length, arranged in the boundary of a square without its rightmost vertical edge. More precisely, denoting the common edge length by $\ell>0$ and this compact three-edge graph by $\cell_\ell$, the corresponding ladder graph  $\L_\ell=(\mathbf{V}_{\L_\ell},\mathbf{E}_{\L_\ell})$ can be identified with
\begin{equation}
	\label{eq-ladderconstr}
	\L_\ell=\bigcup_{k\in\mathbb Z}\cell_\ell^k\,,
\end{equation}
where $\cell_\ell^k$ is a copy of $\cell_\ell$ translated along a fixed axis by a factor $\ell k$, for every $k\in\Z$ (see Figure \ref{fig:ladder}).
Alternatively, it is often convenient to think of $\L_\ell$ as suitably isometrically embedded in $\R^2$, taking for instance
\[
\mathbf{V}_{\L_\ell}\cong \left\{(\ell k,\ell/2): k\in\Z\right\}\cup\left\{(\ell k,-\ell/2):k\in\Z\right\}=:\mathbf{V}_\ell^+\cup\mathbf{V}_\ell^-\subset\R^2,
\]
and
\[
\mathbf{E}_{\L_\ell}\cong\mathbf{E}_\ell^+\cup\mathbf{E}_\ell^-\cup \mathbf{E}_\ell^v\subset\R^2,
\]
where
\begin{gather*}
\mathbf{E}_\ell^+:=\left\{\big(\ell k,\ell(k+1)\big)\times\{\ell/2\}: k\in\Z\right\},\qquad\mathbf{E}_\ell^-:=\left\{\big(\ell k,\ell(k+1)\big)\times\{-\ell/2\}: k\in\Z\right\}, \\[.2cm]
\mathbf{E}_\ell^v:=\left\{\{\ell k\}\times(-\ell/2,\ell/2): k\in\Z\right\},
\end{gather*}
so that one can write $\L_\ell=\mathbf{V}_{\L_\ell}\cup\mathbf{E}_{\L_\ell}$. Similarly, one can also write
\[
 \L_\ell=\mathbf{E}_\ell^v\cup\S_\ell^+\cup\S_\ell^-
\]
where $\S_\ell^{\pm}$ denote the horizontal lines $y=\pm\ell/2$. In the following, with a slight abuse of notation we will often use $\R$ in place of $\S_\ell^{\pm}$ where this does not give rise to misunderstandings.

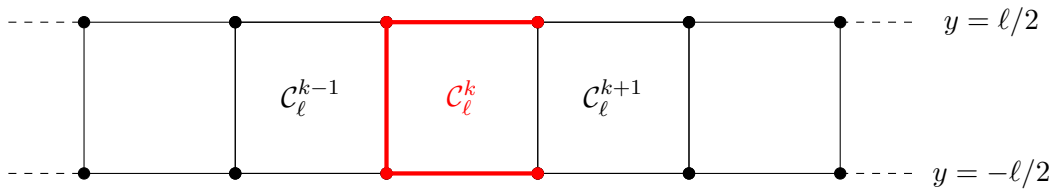
\begin{figure}[!ht]
\centering
\begin{tikzpicture}[scale=1,>=latex]
\draw[step=2] (0,0) grid (10,2);
\draw[dashed] (-1,0)--(0,0);
\draw[dashed] (-1,2)--(0,2);
\draw[dashed] (10,0)--(11,0);
\draw[dashed] (10,2)--(11,2);
\draw[step=2] (2,0) grid (8,2);
\node at (3,1) [infinito] {$\cell_\ell^{k-1}$};
\node at (5,1) [infinito] {$\red{\cell_\ell^{k}}$};
\node at (7,1) [infinito] {$\cell_\ell^{k+1}$};
\foreach \i in {0,...,5}
 \foreach \j in {0,1}
 {\node at (\i*2,\j*2) [nodo] {};}
 \foreach \i in {1,...,4}
 \foreach \j in {0,1}
 {\node at (\i*2,\j*2) [nodo] {};}
\node at (12,2) [infinito] {$y=\ell/2$};
\node at (12,0) [infinito] {$y=-\ell/2$};
\draw[ultra thick, red] (4,2)--(6,2);
\draw[ultra thick, red] (4,0)--(6,0);
\draw[ultra thick, red] (4,0)--(4,2);
\node[red] at (4,0) [nodo] {};
\node[red] at (4,2) [nodo] {};
\node[red] at (6,0) [nodo] {};
\node[red] at (6,2) [nodo] {};
\end{tikzpicture}
\caption{The ladder graph $\L_\ell$ embedded in $\R^2$ and seen as the union of infinitely many copies of the compact graph $\cell_\ell$.}
\label{fig:ladder}
\end{figure}

The main result of this section is the following characterization of the ground state energy level on ladder graphs for small mass.
\begin{proposition}
	\label{thm:ladder}
	Let $p\in(2,6)$. Then
	\begin{equation}\label{eq:Las}
		\lim_{\mu\to0}\frac{\E_{\L_1}(\mu)}{\mu^{\frac{p+2}{6-p}}}=\widetilde\E_{\R}\left(1/3\right)<0\,,
	\end{equation}
	where
	\begin{equation}
		\label{eq:gsE23}
		\widetilde\E_{\R}\left(1/3\right):=\inf_{u \in H^1_{1/3}(\R)} \widetilde E(u,\R),
	\end{equation}
	and
	\begin{equation}\label{eq:E23}
		\widetilde E(u,\R):=\int_\R |u'|^2 \dd x-\frac{3}{p}\int_\R |u|^p \dd x.
	\end{equation}
\end{proposition}
Recall that the existence of normalized ground states on ladder graphs is guaranteed for every $p\in(2,6)$ and every $\mu>0$ by \cite[Theorem 1.1]{Dovetta19}. Similarly, since the minimization problem \eqref{eq:gsE23} is just a slight modification of the usual NLS ground states problem on the line, $\widetilde\E_\R(1/3)$ is clearly attained and it admits a unique positive and radially symmetric decreasing minimizer. Since a standard scaling argument shows that $\widetilde\E_\R(\nu)=\widetilde\E_\R(1)\nu^{(p+2)/(6-p)}$ for every $\nu>0$ and $p\in(2,6)$, Proposition \ref{thm:ladder} actually proves that, at leading order, the energy level of ground states with small masses on the ladder graph $\L_1$ coincides with that of ground states with the same mass for a suitable one-dimensional problem.

Before proving Proposition \ref{thm:ladder}, we need the next preliminary lemma.
\begin{lemma}
	Let $p\in(2,6)$, $\mu>0$, and, for every $\ell>0$, let $u_\ell$ be a ground state with mass $\mu$ on $\L_\ell$. Then for every $r\geq2$,
	\begin{equation}\label{eq:LRid}
		\Vert u_\ell\Vert_{L^r(\S_\ell^+)}=\Vert u_\ell\Vert_{L^r(\S_\ell^-)}+o(1)=\Vert u_\ell\Vert_{L^r(\mathbf{E}_\ell^v)}+o(1)\qquad \text{as}\quad\ell\to 0\,.
	\end{equation}
\end{lemma}
\begin{proof}
	Observe first that there exists $C>0$ depending only on $p$ and $\mu$ such that \begin{equation}
		\label{eq:uderb}
		\E_{\L_\ell}(\mu)\leq -C\qquad\text{and}\qquad \Vert u'_\ell\Vert_{L^2(\L_\ell)}\leq C
	\end{equation}
	for every $\ell>0$ small enough.
	Indeed, let $\varphi\in H^1(\R)$ be a compactly supported function such that $E(\varphi,\R)<0$ and $\|\varphi\|_{L^2(\R)}^2=\mu/3$ (such a function surely exists because $\E_\R(\nu)<0$ for every $\nu>0$ and by density of $C_c^\infty(\R)$ in $H^1(\R)$). Define then $v_\ell\in H^1(\L_\ell)$ by setting $v_\ell^+\equiv v^-_\ell\equiv \varphi$, with $v_\ell^{\pm}:={v_\ell}_{|\S_\ell^{\pm}}$, and taking $v_\ell$ constant on each edge of $\mathbf{E}_\ell^v$. Then, 
	\[
	\|v_\ell\|_{L^2(\L_\ell)}^2=\|v_\ell^+\|_{L^2(S_\ell^+)}^2+\|v_\ell^-\|_{L^2(S_\ell^-)}^2+\sum_{e\in\mathbf{E}_\ell^v}\|v_\ell\|_{L^2(e)}^2=2\|\varphi\|_{L^2(\R)}^2+\ell\sum_{k\in\Z}\varphi(k\ell)^2\,,
	\]
	and since $\displaystyle \ell\sum_{k\in\Z}\varphi(k\ell)^2\to\|\varphi\|_{L^2(\R)}^2$ as $\ell\to0^+$, we have $\|v_\ell\|_{L^2(\L_\ell)}^2\to\mu$ as $\ell\to0^+$. Hence, setting $\kappa_\ell:=\sqrt{\mu/\|v_\ell\|_{L^2(\L_\ell)}}$, we obtain
	\[
	\E_{\L_\ell}(\mu)\leq E(\kappa_\ell v_\ell,\L_\ell)<2E(\kappa_\ell \varphi,\R)=2E(\varphi,\R)+o(1)<0\qquad\text{as }\ell\to0^+\,,
	\]
	which proves the first inequality in \eqref{eq:uderb}. 
	
	To derive the second one, recall first that, like every other connected noncompact metric graph, $\L_\ell$ supports the Gagliardo-Nirenberg inequalities \eqref{eq-gn1d}.
	Moreover, since by the structure of $\L_\ell$ every nonnegative function $u\in H^1(\L_\ell)$ attains almost every value in its image at least twice on $\L_\ell$, the standard rearrangement theory (see e.g. \cite[Section 3]{Adami15}) ensures that
	\[
	\|u'\|_{L^2(\L_\ell)}\geq \|\widehat{u}\,'\|_{L^2(\R)}\,,\qquad \|u\|_{L^r(\L_\ell)}=\|\widehat u\|_{L^r(\R)}\quad\forall r\geq2\,,
	\]
	where $\widehat{u}$ denotes the symmetric rearrangement of $u$ in $H^1(\R)$. Combining with \eqref{eq-gn1d} and the corresponding Gagliardo-Nirenberg inequalities in $H^1(\R)$ immediately shows that the optimal constants for which \eqref{eq-gn1d} holds on $\L_\ell$ are not larger than the optimal constants in the corresponding inequalities on $\R$. Moreover, since for every $\ell>0$ on the real line there exist optimizing sequences for these inequalities given by functions compactly supported on $[0,\ell]$, thinking of such sequences as supported on a single edge of $\L_\ell$ (and extended to zero elsewhere on $\L_\ell$) proves that the optimal constants in \eqref{eq-gn1d} actually coincide with those of the corresponding inequality on $\R$ for every $\ell>0$. In particular, they are independent of $\ell$. Hence, letting $u_\ell\in H_\mu^1(\L_\ell)$ be a ground state with mass $\mu$ on $\L_\ell$, and combining \eqref{eq-gn1d} with the strict negativity of $\E_{\L_\ell}(\mu)$ as given by the first inequality in \eqref{eq:uderb} yields
	\[
	\|u_\ell'\|_{L^2(\L_\ell)}^2\leq \frac2p\|u_\ell\|_{L^p(\L_\ell)}^p\lesssim\mu^{\frac p4+\frac12}\|u_\ell'\|_{L^2(\L_\ell)}^{\frac p2-1}
	\]
	for every $\ell>0$ small enough, in turn implying the second inequality in \eqref{eq:uderb} for every $p\in(2,6)$. 
	
	We now prove that, for every $r\geq2$, 
	 \begin{equation}
	 	\label{eq:u+-u-}
	\Vert u^+_\ell -u^-_\ell\Vert_{L^r(\R)}=o(1),\qquad \text{as}\quad\ell\to 0\,,
	\end{equation}
	where $u_\ell$ is again a ground state with mass $\mu$ on $\L_\ell$ and $u_\ell^{\pm}:={u_\ell}_{|\S_\ell^{\pm}}$. Given $x\in\R$, in view of the embedding of $\L_\ell$ in $\R^2$ and of the identification of $\S_\ell^\pm$ with $\R$, the values of $u_\ell$ at the points $(x,\pm\ell/2)$ coincide with $u_\ell^\pm(x)$. Hence, for every $k\in\Z$,
	\[
	\vert u^+_\ell(x)-u^-_\ell(x)\vert\leq \sqrt{3\ell}\|u_\ell'\|_{L^2(\cell_\ell^k)},\quad\forall x\in[\ell k,\ell(k+1)],
	\]
	so that
	\[
	\begin{split}
		\Vert u^+_\ell-u^-_\ell\Vert^2_{L^2(\R)}&=\int_\R\vert u^+_\ell(x)-u^-_\ell(x)\vert^2\,dx=\sum_{k\in\mathbb Z}\int_{\ell k}^{\ell(k+1)}\vert u^+_\ell(x)-u^-_\ell(x)\vert^2\,dx \\
		&\leq 3\ell^2\sum_{k\in\mathbb Z}\|u_\ell'\|_{L^2(\cell_\ell^k)}^2=3\ell^2\Vert u'_\ell\Vert^2_{L^2(\L_\ell)}=o(1)\qquad\text{as}\quad\ell\to0,
	\end{split}
	\]
	by \eqref{eq:uderb}. This proves \eqref{eq:u+-u-} for $r=2$. The case $r>2$ follows from $r=2$, combined with \eqref{eq-gn1d} and \eqref{eq:uderb}.
	
	Arguing analogously we can also prove that, for every $r\geq2$,
	\begin{equation}
		\label{eq:u+-uv}
		\Vert u^+_\ell -u_\ell^v\Vert_{L^r(\R)}=o(1),\qquad \text{as}\quad\ell\to 0\,,
	\end{equation}
	where $u_\ell^v:={u_{\ell}}_{|\mathbf{E}_\ell^v}$, and we think of it as a function defined on the real line $\R$ identifying the edge $\{k\}\times[-\ell/2,\ell/2)\in\mathbf{E}_\ell^v$ with the interval $(\ell k,\ell(k+1))\in\R$ (strictly speaking, one should still define $u_\ell^v(k\ell)$ for every $k\in\Z$, but since this is not relevant for \eqref{eq:u+-uv}, we omit this detail here). To prove \eqref{eq:u+-uv} it is enough to observe that, given $k\in\Z$ and $x\in(k\ell,\ell(k+1))$, the value of $u_\ell$ at $(x,\ell/2)$ coincides with $u_\ell^+(x)$ and the value of $u_\ell$ at $(\ell k,x-\ell(k+\frac{1}{2}))$ coincides with $u_\ell^v(x)$, so that
	\[
	\vert u^+_\ell(x)-u^v_\ell(x)\vert\leq \sqrt{2\ell}\|u_\ell'\|_{L^2(\cell_\ell^k)}.
	\]
	Since \eqref{eq:LRid} follows directly by combining \eqref{eq:u+-u-} and \eqref{eq:u+-uv}, we conclude.
\end{proof}

\begin{proof}[Proof of Proposition \ref{thm:ladder}]
First, note that, if $\M$ is a general connected metric graph and $u\in H^1(\M)$ is fixed, then $\|u_t\|_{L^2(\M_t)}^2=t\|u\|_{L^2(\M)}^2$ and
\begin{equation}\label{eq:normenergy}
  \frac{E(u,\M)}{\left( \Vert u\Vert^2_{L^2(\M)}\right)^{2\beta+1}}=\frac{E(u_t,\M_t)}{\left( \Vert u_t\Vert^2_{L^2(\M_t)}\right)^{2\beta+1}}\qquad\forall t>0,
\end{equation}
where
\begin{equation}\label{eq:rescaling}
 \M_t:=t^{-\beta}\M,\qquad u_t(x):=t^{\alpha}u(t^\beta x)\quad\forall x\in \M_t,
\end{equation}
and
\begin{equation}\label{eq:ab}
 \alpha:=\frac{2}{6-p}\,,\qquad\beta:=\frac{p-2}{6-p}\,.
 \end{equation}
As a consequence, letting $\M=\L_1$ and fixing $u\in H^1_\mu(\L_1)$ and $t=t(\mu):=\mu^{-1}$, one finds that $\|u_{t(\mu)}\|_{L^2(\L_{\ell(\mu)})}^2=1$ and that
\begin{equation}\label{eq:normenunit}
 \frac{E(u,\L_1)}{\mu^{\frac{p+2}{6-p}}}=E(u_{t(\mu)},\L_{\ell(\mu)}),\qquad\text{with}\quad \ell(\mu):=\mu^{\beta}.
\end{equation}
As a consequence, since $\ell(\mu)\to0$ as $\mu\to0$, in order to prove \eqref{eq:Las} it is sufficient to prove that
\begin{equation}\label{eq:newclaim}
 \lim_{\ell\to 0} \E_{\L_\ell}(1)=\widetilde\E_{\R}\left(1/3\right).
\end{equation}
To this aim, let $U$ be the unique positive and radially symmetric decreasing minimizer of $\widetilde\E_\R(1/3)$. Also, fix a function $w_\ell \in H^1(\L_\ell)$ that satisfies $w^\pm_\ell:={w_\ell}_{\S_\ell^\pm}=U$ and is constant on each edge of $\mathbf{E}_\ell^v$. By construction,
\[
\Vert w_\ell\Vert^r_{L^r(\mathbf{E}_\ell^v)}=\ell\sum_{k\in\mathbb Z} U^r(k\ell)=\Vert U\Vert^r_{L^r(\R)}+o(1)\qquad\text{as}\quad\ell\to0,
\]
for every $r\geq 2$. Hence, letting $\kappa_\ell\in\R$ be such that $\widetilde w_\ell:=\kappa_\ell w_\ell \in H^1_1(\L_\ell)$, we have $\kappa_\ell\to 1$ and
    \[
    \Vert \widetilde w_\ell\Vert^p_{L^p(\L_\ell)}=3\Vert U\Vert^p_{L^p(\R)}+o(1),\qquad \Vert \widetilde w'_\ell\Vert^2_{L^2(\L_\ell)}=2\Vert U'\Vert^2_{L^2(\R)}+o(1),\qquad\text{as}\quad\ell\to0,
    \]
    so that
    \begin{equation}
    	\label{eq-first_est}
    \E_{\L_\ell}(1)\leq E(\widetilde w_\ell, \L_\ell)=\widetilde E(U,\R)+o(1)=\widetilde\E_{\R}\left(1/3\right)+o(1)\qquad \text{as}\quad\ell\to0.
    \end{equation}
    Let now $u_\ell\in H^1_1(\L_\ell)$ be a ground state on $\L_\ell$. Without loss of generality we can assume $u_\ell>0$ on $\L_\ell$ and
    \begin{equation}
    	\label{eq-der_ord}
    \Vert(u^+_\ell)' \Vert_{L^2(\R)}\geq   \Vert(u^-_\ell)' \Vert_{L^2(\R)},\qquad\text{for every}\quad\ell>0.
    \end{equation}
    Letting $\sigma_\ell\in\R$ be such that $\sigma_\ell u^-_\ell\in H^1_{1/3}(\R)$, one finds that $\sigma_\ell\to 1$, by \eqref{eq:LRid}, and that $\widetilde E(\sigma_\ell u^-_\ell,\R)=\widetilde E(u^-_\ell,\R)+o(1)$, by \eqref{eq:uderb} and \eqref{eq-gn1d}, as $\ell\to 0$. Thus, by \eqref{eq-der_ord} and again \eqref{eq:LRid},
    \[
    \widetilde\E_{\R}\left(1/3\right)\leq \widetilde E(\sigma_\ell u^-_\ell, \R)= \widetilde E(u^-_\ell, \R)+o(1)\leq E(u_\ell, \L_\ell)+o(1)=\E_{\L_\ell}(1)+o(1),
    \]
    which, combined with \eqref{eq-first_est}, yields \eqref{eq:newclaim} and concludes the proof.
      \end{proof}
 

\section{The half-grid and the quarter-grid: proof of Theorem \ref{thm:half-quart}}
\label{sec-halfquart}

In this section we present the proof of Theorem \ref{thm:half-quart}, which shows the existence of two examples of defected grids giving rise to the dimensional crossover from the point of view of normalized ground states.

First, we briefly sketch some well-known preliminary remarks, which are valid even on more general defected grids, that prove \eqref{eq-energylevelHQ} and the nonexistence claims of the theorem. Then, we prove the existence results and the estimates \eqref{eq-massrelationHQp}--\eqref{eq-massrelationHQ4}--\eqref{eq-energyrelationH}--\eqref{eq-energyrelationQ}. Here, it is convenient to divide the case of the half-grid and the case of the quarter-grid, since the proof for the latter relies heavily on the former. Finally, we conclude the section by establishing some auxiliary results on some modified defected grids, that are necessary for the proofs of the subsequent main results of the paper.


\subsection{Preliminary estimates and nonexistence results}
\label{subsec-nonexboth}

Let us start from some general features that have been proved in \cite[Section 4]{Dovetta22} and are valid under very broad assumptions.

First, if $X=\HH$ or $X=\Q$, then $\E_X(\mu)\leq0$, since the bottom of the spectrum of the operator $-d^2/dx^2$ endowed with homogeneous Kirchhoff vertex conditions on $X$ is zero (see, e.g. \cite[Section 4.1.1]{Dovetta22}). Moreover, $X$ supports \eqref{eq-gn1d}, \eqref{eq-gn2d} and \eqref{eq-gnint}. As a consequence, the NLS energy functional $E(\cdot,X)$ is bounded from below and coercive on $H_\mu^1(X)$ for every $p\in(2,6)$ and every $\mu>0$. Thus, any minimizing sequence $(u_n)_n\subset H_\mu^1(X)$ is bounded in $H^1(X)$ and, up to subsequences,
\begin{equation}
	\label{eq-minlimit}
u_n\rightharpoonup u\quad\text{in}\:H^1(X),\qquad u_n\to u\quad\text{in}\:L^\infty_{\mathrm{loc}}(X),
\end{equation}
so that, by \cite[Lemma 4.8]{Dovetta22}, either $u\equiv0$ on $X$, or $u$ is a ground state with mass $\mu$ on $X$. In addition, in a very classical way, one can check that the function $\E_X(\cdot)$ is continuous, non-increasing, concave and subadditive on $[0,+\infty)$.

For every $p\in[4,6)$, denoting by $K(p,X)$ the optimal constant of \eqref{eq-gnint} on $X$,
\[
E(u,X)\geq\frac12\|u'\|_{L^2(X)}^2\left(1-\frac{2K(p,X)}{p}\,\mu^{\frac{p-2}{2}}\right),\qquad\forall u\in H_\mu^1(X),
\]
so that
\begin{equation}
	\label{eq-energypos}
E(u,X)\geq0,\qquad\forall u\in H_\mu^1(X),\quad\forall \mu\leq\mu_{p,X}:=\left(\frac{p}{2K(p,X)}\right)^{\frac{2}{p-2}}.
\end{equation}
Hence, the first line of \eqref{eq-energylevelHQ} is proved. Moreover, letting $\mu>\mu_{p,X}$ and $(u_n)_n\subset H_\mu^1(X)$ be a maximizing sequence for $K(p,X)$, i.e.
\[
\frac{\|u_n\|_{L^p(X)}^p}{\|u_n'\|_{L^2(X)}^2}\to K(p,X)\mu^{\frac{p-2}{2}}\qquad\text{as }n\to+\infty,
\] 
we obtain, for $n$ sufficiently large,
\[
\E_X(\mu)\leq E(u_n,X)=\frac12\|u_n'\|_{L^2(X)}^2\left(1-\left(\frac\mu{\mu_{p,X}}\right)^{\frac{p-2}{2}}+o(1)\right)<0\,,
\]
completing the proof of \eqref{eq-energylevelHQ}. In addition, since \eqref{eq-energypos} is strict whenever $\mu<\mu_{p,X}$, one has the nonexistence claims of \emph{2)a)} and \emph{2)b)} in Theorem \ref{thm:half-quart}.


\subsection{Existence on the half-grid $\HH$}
\label{subsec-halfgrid}

In view of the previous subsection, when $X=\HH$ it remains to prove the existence claims, the latter inequalities in \eqref{eq-massrelationHQp} and \eqref{eq-massrelationHQ4}, and \eqref{eq-energyrelationH}.

As a first step, we establish the following existence criterion for ground states.

\begin{lemma}[Existence criterion on $\HH$]
	\label{lem-critH}
	Let $p\in(2,6)$ and $\mu>0$. If $\E_\HH(\mu)<\E_\Gamma(\mu)$, then there exists a ground state with mass $\mu$ on $\HH$.
\end{lemma}

\begin{proof}
	Let $(u_n)_n\subset H_\mu^1(\HH)$ be a minimizing sequence, that we can consider nonnegative without loss of generality. Let also $u$ be its limit, up to subsequences, as in \eqref{eq-minlimit}. We have to prove that $u\not\equiv0$. Assume by contradiction that $u_n\to0$ in $L_{\mathrm{loc}}^\infty(\HH)$. Fix a closed edge $e\in\partial\HH$ (in Figure {\sc \ref{fig:halfgrid_a}} $\partial\HH$ is the rightmost vertical straight line). Since $\HH$ is invariant by discrete translations in the vertical direction, it is not restrictive to assume that, for every $n$ there exists $x_n\in e$ such that $\displaystyle u_n(x_n)=\max_{\partial\HH}u_n$. Hence, $u_n\to0$ in $L^\infty(e)$ and thus, by construction, $u_n\to0$ in $L^\infty(\partial\HH)$. Note that $\|u_n\|_{L^\infty(\HH)}>\|u_n\|_{L^\infty(\partial\HH)}$ as $n\to+\infty$. Indeed, if this were not the case, $u_n\to0$ in $L^\infty(\partial\HH)$ would yield $u_n\to0$ in $L^\infty(\HH)$, and thus $u_n\to0$ in $L^p(\HH)$ as $n\to+\infty$ (since $\|u_n\|_{L^2(\HH)}^2=\mu$ for every $n$). In turn, this would imply 
	\[
	\E_{\HH}(\mu)=\lim_n E(u_n,\HH)=\lim_n\frac{\|u_n'\|_{L^2(\HH)}^2}2\geq0\,,
	\] 
	thus contradicting the hypothesis that $\E_{\HH}(\mu)<\E_{\Gamma}(\mu)$, since $\E_{\Gamma}(\mu)\leq0$ for every $\mu>0$ and $p\in(2,6)$.

	Hence, setting $\ep_n:=\|u_n\|_{L^\infty(\partial\HH)}$, we have $(u_n-\ep_n)_+\not\equiv0$ on $\HH$ for large $n$. Defining $v_n:=\kappa_n(u_n-\ep_n)_+$, with $\kappa_n>0$ such that $\|v_n\|_{L^2(\HH)}^2=\mu$, it is readily seen that $E(v_n,\HH)\leq E(u_n,\HH)+o(1)$ for large $n$. Moreover, since ${v_n}_{|\partial\HH}\equiv0$ by construction, we can actually think of $v_n$ as a function on the whole grid $\Gamma$, simply identifying $\HH$ with a corresponding subset of $\Gamma$ and extending $v_n$ to zero on $\Gamma\setminus\HH$. As a consequence, we obtain
	\[
	\E_\HH(\mu)=\lim_n E(u_n,\HH)\geq\lim_n E(v_n,\Gamma)\geq\E_\Gamma(\mu),
	\]
	which contradicts again the hypothesis that  $\E_{\HH}(\mu)<\E_{\Gamma}(\mu)$.
\end{proof}

\begin{proof}[Conclusion of the proof of Theorem \ref{thm:half-quart} for $X=\HH$]
 Consider first the cases where ground states with mass $\mu$ exist on $\Gamma$ (see \cite[Theorems 1.1 and 1.2]{Adami19b}): $p\in(2,4)$ and $\mu>0$, $p\in(4,6)$ and $\mu\geq\mu_{p,\Gamma}$, $p=4$ and $\mu>\mu_{4,\Gamma}$. Let $u$ be a ground state with mass $\mu$ on $\Gamma$. By \cite[Proposition 4.2]{Dovetta22}, there exists a compact set $\K\subset\Gamma$ such that $E(u,e)>0$, for every edge $e\in\Gamma\setminus\K$. Moreover, by the periodicity of $\Gamma$, it is not restrictive to assume that $\K\in\HH\setminus\partial\HH$ (where we are again thinking of $\HH$ as a subset of $\Gamma$). Now, let $\kappa:=\sqrt{\mu/(\mu-\|u\|_{L^2(\Gamma\setminus\HH)}^2)}$, so that $v:=\kappa u_{|\HH}$ belongs to $H_\mu^1(\HH)$. Hence, using that $\kappa>1$ and that $\E_\Gamma(\mu)\leq0$, we obtain
 \[
 \E_\HH(\mu)\leq E(v,\HH)\leq\kappa^2E(u,\HH)<\kappa^2E(u,\Gamma)\leq E(u,\Gamma)=\E_\Gamma(\mu),
 \]
 which implies existence of ground states on $\HH$ in view of Lemma \ref{lem-critH}. This proves \emph{1)}, the second inequalities in \eqref{eq-massrelationHQp} and \eqref{eq-massrelationHQ4}, the existence results in \emph{2)a)}
 and \emph{2)b)} for $\mu\geq\mu_{p,\Gamma}$ and $\mu>\mu_{4,\Gamma}$ (respectively) and the first line of \eqref{eq-energyrelationH} up to the cases that we manage below.
 
 Now, if either $p\in(4,6)$ and $\mu\in(\mu_\HH,\mu_\Gamma)$, or $p=4$, $\mu_{4,\HH}<\mu_{4,\Gamma}$ and $\mu\in(\mu_{4,\HH},\mu_{4,\Gamma}]$, then $\E_\HH(\mu)<0=\E_\Gamma(\mu)$, whence one gets again the existence of ground states and the first line of \eqref{eq-energyrelationH}. To conclude, it remains to deal with the case $p\in(4,6)$ and $\mu=\mu_{p,\HH}$, where $\E_\HH(\mu)=\E_\Gamma(\mu)=0$. Let then $(u_n)_n\subset H_{\mu_{p,\HH}}^1(\HH)$ be a maximizing sequence for the Gagliardo-Nirenberg inequality \eqref{eq-gnint}, i.e.
 \begin{equation}
 \label{eq:optGN}
 \frac{\|u_n\|_{L^p(\HH)}^p}{\mu_{p,\HH}^{\frac{p-2}{2}}\|u_n'\|_{L^2(\HH)}^2}\to K(p,\HH)\qquad\text{as }n\to+\infty.
 \end{equation}
 Exploiting once more the invariance of $\HH$, without loss of generality we can further assume that $\|u_n\|_{L^\infty(\partial\HH)}$ is attained in the same fixed edge of $\partial\HH$ for every $n$. Now, combining \eqref{eq:optGN} with \eqref{eq-gn1d} gives
 \[
 K(p,\HH)+o(1)\lesssim_{p,\HH}\frac{\mu_{p,\HH}^{\frac p4+\frac12}\|u_n'\|_{L^2(\HH)}^{\frac p2-1}}{\mu_{p,\HH}^{\frac{p-2}{2}}\|u_n'\|_{L^2(\HH)}^2}=\frac{\mu_{p,\HH}^{\frac{6-p}{4}}}{\|u_n'\|_{L^2(\HH)}^{\frac{6-p}{2}}}
 \]
 for $n$ large, so that $(u_n)_n$ is bounded in $H^1(\HH)$. Hence, up to subsequences $u_n\rightharpoonup u$ in $H^1(\HH)$ as $n\to+\infty$, for some $u\in H^1(\HH)$. Furthermore, 
 \[
 E(u_n,\HH)=\frac12\|u_n'\|_{L^2(\HH)}^2\left(1-\frac{2K(p,\HH)}{p}\mu_{p,\HH}^{\frac{p-2}{2}}+o(1)\right)\to0\qquad\text{as }n\to+\infty\,,
 \]
 namely $(u_n)_n$ is a minimizing sequence for $E$ in $H_{\mu_{p,\HH}}^1(\HH)$. As already recalled before, by \cite[Lemma 4.8]{Dovetta22} we then have either $u\equiv0$ on $\HH$, or $u$ is a ground state with mass $\mu_{p,\HH}$ on $\HH$. 
 
 Assume by contradiction that $u\equiv0$. Since by assumption $\|u_n\|_{L^\infty(\partial\HH)}$ is attained always in the same edge, the $L^\infty_{\text{\normalfont loc}}$ convergence of $u_n$ to $u$ yields $u_n\to0$ in $L^\infty(\partial\HH)$. On the contrary, using \eqref{eq-gn2d} or \eqref{eq-gnint} with $p=4$,
 \[
 \begin{split}
 K(p,\HH)+o(1)=\frac{\|u_n\|_{L^p(\HH)}^p}{\mu_{p,\HH}^{\frac{p-2}{2}}\|u_n'\|_{L^2(\HH)}^2}\leq\frac{\|u_n\|_{L^\infty(\HH)}^{p-4}\|u_n\|_{L^4(\HH)}^4}{\mu_{p,\HH}^{\frac{p-2}{2}}\|u_n'\|_{L^2(\HH)}^2}&\,\lesssim_{\HH}\|u_n\|_{L^\infty(\HH)}^{p-4}\frac{\mu_{p,\HH}\|u_n'\|_{L^2(\HH)}^2}{\mu_{p,\HH}^{\frac{p-2}{2}}\|u_n'\|_{L^2(\HH)}^2}\\
 &\,=\|u_n\|_{L^\infty(\HH)}^{p-4}\mu_{p,\HH}^{2-\frac p2}\,,
 \end{split}
 \]
 so that $\|u_n\|_{L^\infty(\HH)}\not\to0$ as $n\to+\infty$. Hence, setting $\varepsilon_n:=\|u_n\|_{L^\infty(\partial\HH)}$ and $v_n:=\kappa_n(u_n-\varepsilon_n)_+$, with $\kappa_n>0$ so that $\|v_n\|_{L^2(\HH)}^2=\mu_{p,\HH}$, we obtain $E(v_n,\HH)\to0$ and  $\|v_n\|_{L^\infty(\HH)}\not\to0$ as $n\to+\infty$, and $v_n\equiv0$ on $\partial\HH$ for every $n$. As we did in the proof of Lemma \ref{lem-critH}, we can interpret again $v_n$ as a function in $H_{\mu_{p,\HH}}^1(\Gamma)$. Furthermore, exploiting the periodicity of $\Gamma$, up to a discrete translation we can assume that $\|v_n\|_{L^\infty(\Gamma)}$ is attained in the same edge of $\Gamma$ for every $n$. Since the $L^\infty$ norm of $v_n$ is bounded away from zero uniformly in $n$, we then have that there exists $v\in H^1(\Gamma)$, $v\not\equiv0$, such that up to subsequences $v_n\rightharpoonup v$ in $H^1(\Gamma)$ and $v_n\to v $ in $L_{\text{\normalfont loc}}^\infty(\Gamma)$ as $n\to+\infty$. Now, if $\displaystyle v_n\to v$ in $L^2(\Gamma)$, then $v\in H_{\mu_{p,\HH}}^1(\Gamma)$ and $v_n\to v$ in $L^p(\Gamma)$, so that by weak lower semicontinuity we would have $E(v,\Gamma)\leq\lim_n E(v_n,\Gamma)=0$, which is impossible since $\E_{\Gamma}(\mu_{p,\HH})=0$ but it is not attained. Similarly, if $v_n\not\to v$ in $L^2(\Gamma)$, then $\displaystyle \mu_{p,\HH}>\liminf_n\|v_n-v\|_{L^2(\Gamma)}^2>0$, and arguing as in the proof of \cite[Lemma 4.8]{Dovetta22} we would obtain
 \[
 \lim_n E(v_n,\Gamma)=E(v,\Gamma)+\lim_n E(v_n-v,\Gamma)\geq E(v,\Gamma)>0\,,
 \]
 where we used again that $\E_{\Gamma}(\nu)=0$ and it is not attained for every $\nu<\mu_{p,\Gamma}$. Since this contradicts the fact that $\displaystyle \lim_n E(v_n,\Gamma)=0$, we conclude.
  \end{proof}


\subsection{Existence on the quarter-grid $\Q$}
\label{subsec-quartgrid}

Here we prove the existence claims on $\Q$, the first inequalities in \eqref{eq-massrelationHQp}--\eqref{eq-massrelationHQ4}, and \eqref{eq-energyrelationQ}.

We start again by establishing an existence criterion for ground states, which, in this case, is a bit more technically demanding since $\Q$ has no integer translation invariance. 

\begin{definition}
	\label{def:Nv}
 Let $\v\in\Q\setminus\partial\Q$ (with $\partial\Q$ as in Figure {\sc \ref{fig:boundary}}) and $a\in\N\setminus\{0\}$ such that $a\leq\mathrm{dist}(\v,\partial\Q)$. We denote by $\NN_a(\v)$ the subgraph of $\Q$ contained in the closed square of side $2a$ centered at $\v$ (see Figure {\sc \ref{fig:square}}, and Figure {\sc \ref{fig:boundarysquare}} for its boundary $\partial\NN_a(\v)$).
\end{definition}

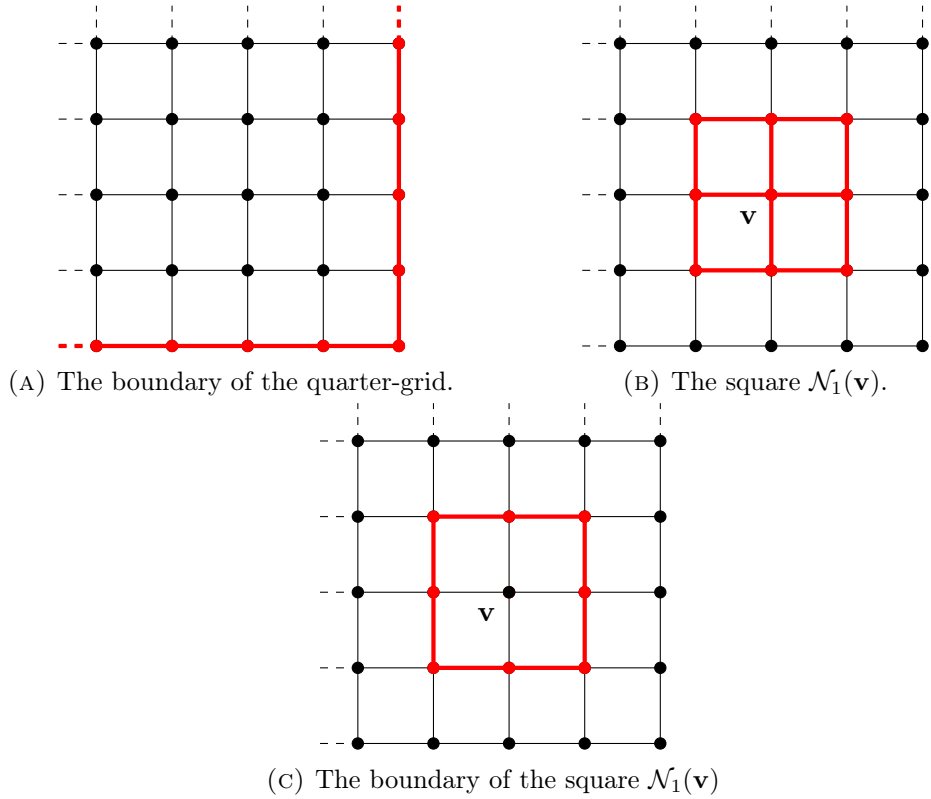
\begin{figure}[htbp]
\centering
\begin{subfigure}[b]{0.48\textwidth}
\centering
\begin{tikzpicture}[scale=1,>=latex]
\draw[step=1] (0,0) grid (4,4);
\foreach \j in {0,...,4}{
\draw[dashed] (-.5,\j)--(0,\j);}
\foreach \i in {0,...,4}{
\draw[dashed] (\i,4)--(\i,4.5);}
\foreach \i in {0,...,4}
\foreach \j in {0,...,4}
{\node at (\i,\j) [nodo] {};}
\draw[ultra thick,red] (0,0)--(4,0);
\draw[ultra thick,red] (4,0)--(4,4);
\draw[ultra thick,dashed,red] (-.5,0)--(0,0);
\draw[ultra thick,dashed,red] (4,4)--(4,4.5);
\foreach \i in {0,...,4}{
\node[red] at (\i,0) [nodo] {};
\node[red] at (4,\i) [nodo] {};}
\end{tikzpicture}
\caption{The boundary of the quarter-grid.}
\label{fig:boundary}
\end{subfigure}
\hspace{-1cm}
\begin{subfigure}[b]{0.48\textwidth}
\centering
\begin{tikzpicture}[scale=1,>=latex]
\draw[step=1] (0,0) grid (4,4);
\foreach \j in {0,...,4}{
\draw[dashed] (-.5,\j)--(0,\j);}
\foreach \i in {0,...,4}{
\draw[dashed] (\i,4)--(\i,4.5);}
\foreach \i in {0,...,4}
\foreach \j in {0,...,4}
{\node at (\i,\j) [nodo] {};}
\draw[ultra thick,red,step=1] (1,1) grid (3,3);
\foreach \i in {1,...,3}
\foreach \j in {1,...,3}
{\node[red] at (\i,\j) [nodo] {};}
\node at (2-.3,2-.3) [infinito] {$\v$};
\end{tikzpicture}
\caption{The square $\NN_1(\v)$.}
\label{fig:square}
\end{subfigure}
\hspace{-1cm}
\begin{subfigure}[b]{0.48\textwidth}
\centering
\begin{tikzpicture}[scale=1,>=latex]
\draw[step=1] (0,0) grid (4,4);
\foreach \j in {0,...,4}{
\draw[dashed] (-.5,\j)--(0,\j);}
\foreach \i in {0,...,4}{
\draw[dashed] (\i,4)--(\i,4.5);}
\foreach \i in {0,...,4}
\foreach \j in {0,...,4}
{\node at (\i,\j) [nodo] {};}
\draw[ultra thick,red,step=1] (1,1) rectangle (3,3);
\foreach \i in {1,...,3}
\foreach \j in {1,...,3}
{\node[red] at (\i,\j) [nodo] {};}
\node at (2-.3,2-.3) [infinito] {$\v$};
\node at (2,2) [nodo] {};
\end{tikzpicture}
\caption{The boundary of the square $\NN_1(\v)$}
\label{fig:boundarysquare}
\end{subfigure}
\caption{The boundary of $\Q$ and the squares surrounding vertices.}
\label{fig:definitionobject}
\end{figure}

\begin{lemma}[Existence criterion on $\Q$]
	\label{lem-critQ}
	Let $p\in(2,6)$ and $\mu>0$. If $\E_\Q(\mu)<\E_\HH(\mu)$, then there exists a ground state with mass $\mu$ on $\Q$.
\end{lemma}

\begin{proof}
As in the proof of Lemma \ref{lem-critH}, it is sufficient to prove that, given a nonnegative minimizing sequence $(u_n)_n$ on $\Q$, if $u_n\to0$ in $L^\infty_{\mathrm{loc}}(\Q)$, then one obtains a contradiction. We denote by $\D$ the set of the vertices $(\w_i)_{i\in\N}$ on the diagonal of $\Q$ highlighted in Figure {\sc \ref{fig:quartgrid_diagonal}}.

Now, for every $n\in\N$ define
\[
M_n:=\sup_{i\in\N}u_n(\w_i),
\]
which is finite and attained since $u_n\in H^1(\Q)$, for every $n\in\N$. Then, we distinguish between two cases:
\begin{enumerate}[label=(\roman*)]
	\item $\displaystyle\lim_n M_n=0$;
	\item $\displaystyle\lim_n M_n=c>0$, up to subsequences.
\end{enumerate}
\smallskip

\emph{Case (i): $\displaystyle \lim_n M_n=0$.} Here one has that $\|{u_n}_{|\D}\|_{\ell^\infty}\to0$. Moreover, $\|u_n\|_{L^\infty(\Q)}\not\to0$, since if this were the case we would have $u_n\to0$ in $L^p(\Q)$, and thus $\displaystyle\E_{\Q}(\mu)=\lim_n E(u_n,\Q)\geq0$, which is impossible since $\E_\Q(\mu)<\E_\HH(\mu)$ and we already proved that $\E_\HH(\mu)\leq0$ for every $\mu>0$ and every $p\in(2,6)$. Thus, setting $\ep_n:=\|{u_n}_{|\D}\|_{\ell^\infty}$ and defining $v_n:=\kappa_n(u_n-\ep_n)_+$, with $\kappa_n>0$ such that $v_n\in H_\mu^1(\Q)$, we obtain a sequence $(v_n)_n\subset H_\mu^1(\Q)$ such that $\displaystyle \lim_nE(v_n,\Q)=\E_\Q(\mu)$ and $v_n\equiv0$ on $\D$.

Consider now the following construction. Take the portion of $\Q$ below the diagonal line passing through the vertices of $\D$ and rotate it 90 degrees counterclockwise about the vertex $\w_0$. This produces a new graph $\Qo$ as shown in Figure {\sc\ref{fig:quartsplit}}, where each point of $\Q$ appears exaclty once, with the sole exception of the vertices $\w_i$ on the diagonal of $\Q$, that in $\Qo$ come in two copies $\w_i^{\pm}$, for every $i\in\N\setminus\{0\}$. As a consequence, any function defined on $\Q$ that vanishes on $\D$ can be regarded as a function on $\Qo$. Furthermore, each such function can be extended to a function actually defined on $\HH$, simply noting that $\Qo$ can be embedded in $\HH$ as in Figure {\sc\ref{fig:quartsplit}}, and setting the function constantly equal to zero on $\HH\setminus\Qo$. 

\begin{figure}[htbp]
\centering
\begin{subfigure}[b]{0.48\textwidth}
\centering
\begin{tikzpicture}[scale=1,>=latex]
\draw[step=1] (0,0) grid (4,4);
\foreach \j in {0,...,4}{
\draw[dashed] (-.5,\j)--(0,\j);}
\foreach \i in {0,...,4}{
\draw[dashed] (\i,4)--(\i,4.5);}
\foreach \i in {0,...,4}
\foreach \j in {0,...,4}
{\node at (\i,\j) [nodo] {};}
\foreach \i in {0,...,4}
{\node at (4-\i-.3,\i-.3) [infinito] {$\w_{\i}$};}
\draw[dashdotted] (4,0)--(-.5,4.5);
\node at (2,-2.5) [infinito] {};
\end{tikzpicture}
\caption{The set of vertices $\D:=(\mathbf{W}_i)_{i\in\N}$ on the diagonal of $\Q$.}
\label{fig:quartgrid_diagonal}
\end{subfigure}
\hspace{.0cm}
\begin{subfigure}[b]{0.48\textwidth}
\centering
\begin{tikzpicture}[scale=1,>=latex]
\foreach \i in {1,...,4}{
\draw (\i,4-\i)--(\i,4);}
\foreach \j in {1,...,4}{
\draw (4-\j,\j)--(4,\j);}
\foreach \i in {0,...,4}{
\draw[dashed] (\i,4)--(\i,4.5);}
\foreach \i in {0,...,4}
\foreach \j in {\i,...,4}
{\node at (4-\i,\j) [nodo] {};}
\foreach \i in {1,...,4}
{\node at (4-\i-.3,\i-.3) [infinito] {$\w_{\i}^+$};}
\node at (4-.3,-.3) [infinito] {$\w_{0}$};
\foreach \i in {1,...,4}{
\draw (\i,-4+\i)--(\i,-4);}
\foreach \j in {1,...,4}{
\draw (4-\j,-\j)--(4,-\j);}
\foreach \i in {0,...,4}{
\draw[dashed] (\i,-4)--(\i,-4.5);}
\foreach \i in {0,...,4}
\foreach \j in {\i,...,4}
{\node at (4-\i,-\j) [nodo] {};}
\foreach \i in {1,...,4}
{\node at (4-\i-.3,-\i+.3) [infinito] {$\w_{\i}^-$};}
\foreach \i in {0,...,3}{
\draw[dotted] (\i,\i-4)--(\i,4-\i);}
\foreach \j in {0,...,4}{
\draw[dotted] (-.5,\j)--(4-\j,\j);
\draw[dotted] (-.5,-\j)--(4-\j,-\j);}
\end{tikzpicture}
\caption{The defected grid $\Qo$ obtained by suitably splitting $\Q$ (dotted the embedding into $\HH$).}
\label{fig:quartsplit}
\end{subfigure}
\caption{From $\Q$ to $\Qo$.}
\label{fig:quarteropen}
\end{figure}
Since none of these operations modifies the mass and the energy of the starting function on $\Q$, applying it to $v_n$ as above yields
\[
\E_\Q(\mu)=\lim_n E(v_n,\Q)=\lim_n E(v_n,\Qo)=\lim_n E(v_n,\HH)\geq\E_\HH(\mu),
\]
which contradicts the assumption $\E_\Q(\mu)<\E_\HH(\mu)$ and completes the discussion of Case {\em(i)}.
\smallskip

\emph{Case (ii): $\displaystyle \lim_n M_n=c>0$, up to subsequences.}

Fix $\varepsilon\in(0,c/2)$. For every $n$, let $\underline{i}_{n,\varepsilon}, \overline{i}_{n,\varepsilon}$ be respectively the smallest and the largest indices for which it holds $u_n(\w_{\underline{i}_{n,\varepsilon}})\geq\varepsilon$, $u_n(\w_{\overline{i}_{n,\varepsilon}})\geq\varepsilon$. Both $\underline{i}_{n,\varepsilon}$ and $\overline{i}_{n,\varepsilon}$ are well-defined since $u_n\in H^1(\Q)$, and $\displaystyle\lim_n \underline{i}_{n,\varepsilon}=+\infty$ since by assumption $u_n\to0$ in $L_{\text{\normalfont loc}}^{\infty}(\Q)$. Consider then the set of indices 
\[
I_n:=\N\cap\left\{x\in\R\,:\,\frac{\underline{i}_{n,\varepsilon}}{6}\leq x \leq\frac{\underline{i}_{n,\varepsilon}}{3}\right\},
\]
and define, for every $k\in I_n$,  
\[
\NN_k:=\bigcup_{\substack{i\in\N \\ \underline{i}_{n,\varepsilon}\leq i \leq \overline{i}_{n,\varepsilon}}}\NN_k(\w_i)
\]
(we omit to highlight the dependence of $I_n$ and $\NN_k$ on $\varepsilon$ to simplify notation). Hence, $\NN_k$ is the subgraph of $\Q$ contained in the union of the squares with edge $2k$ centered at vertex $\w_i$ on the diagonal $\D$ between $\w_{\underline{i}_{n,\varepsilon}}$ and $\w_{\overline{i}_{n,\varepsilon}}$. Note that, by construction, $\text{\normalfont dist}(\NN_k,\partial\Q)>0$ for every $k$.

Observe that $\partial N_{k_1}\cap\partial N_{k_2}=\emptyset$ for every $k_1\neq k_2$, because $\partial N_{k-1}(\w_i)\subset N_k(\w_i)\setminus\partial\NN_k(\w_i)$ for every $i, k\in\N$ by Definition \ref{def:Nv}. Furthermore, by construction $\NN_k(\w_{i})\setminus \NN_k(\w_{i+1})$ always contains the lower horizontal edge and the rightmost vertical edge of $\partial\NN_k(\w_i)$. Hence, $\partial \NN_k$ always contains the lower horizontal edge and the rightmost vertical edge of $\partial\NN_k(\w_{\underline{i}_{n,\varepsilon}})$. Since $k\geq \underline{i}_{n,\varepsilon}/6$, this ensures that
\[
\left|\partial\NN_k\right|\geq \frac23\underline{i}_{n,\varepsilon}\qquad\forall k\in I_n.
\]
We now claim that there exists a sequence $(k_n)_n$ such that $k_n\to+\infty$ and $\|u_n\|_{L^\infty(\partial\NN_{k_n})}\to0$ as $n\to+\infty$. Indeed, assume by contradiction that this is not the case, i.e. there exists $\beta>0$ such that $\|u_n\|_{L^\infty(\partial\NN_k)}\geq\beta$ for every $k\in I_n$ and every $n$. Let then the indices in $I_n$ be reordered so that $\|u_n'\|_{L^2(\partial\NN_k)}$ is increasing in $k\in I_n$, and let $\widetilde I_n$ be the set collecting the first half of the indices in $I_n$ after the reordering. Since the total number of indices in $I_n$, is at least $\frac{\underline{i}_{n,\varepsilon}}6-2$, we have
\[
\|u_n'\|_{L^2(\Q)}^2\geq\sum_{k\in I_n}\|u_n'\|_{L^2(\partial\NN_k)}^2\geq \frac12\left(\frac{\underline{i}_{n,\varepsilon}}6-4\right)\max_{k\in\widetilde I_n}\|u_n'\|_{L^2(\partial\NN_k)}^2\,,
\]
where we used the fact that $\partial\NN_k$ are disjoint for different values of $k$, that $\displaystyle\|u_n'\|_{L^2(\partial\NN_j)}^2\geq\max_{k\in\widetilde I_n}\|u_n'\|_{L^2(\partial\NN_k)}^2$ for every $j\in I_n\setminus \widetilde I_n$, and that $\displaystyle \#I_n\setminus\widetilde I_n\geq \frac12\left(\frac{\underline{i}_{n,\varepsilon}}6-4\right)$. Since $(u_n)_n$ is bounded in $H^1(\Q)$, this implies that there exists a universal constant $C>0$ such that 
\begin{equation}
	\label{eq:boundu'}
\|u_n'\|_{L^2(\partial\NN_k)}^2\leq \frac{C}{\underline{i}_{n,\varepsilon}}\qquad\forall k\in\widetilde I_n.
\end{equation}
Now, for any such indices, take an edge $\partial\NN_k$ where $\|u_n\|_{L^\infty(\partial\NN_k)}$ is attained. Since by assumption $\|u_n\|_{L^\infty(\partial\NN_k)}\geq\beta$, by \eqref{eq:boundu'} we obtain $|u_n(y)|\geq\beta-\sqrt{C\underline{i}_{n,\varepsilon}^{-1}}\ge\beta/2$ for every $y$ on the considered edge and for every $n$ large enough. In turn, this yields
\[
\|u_n\|_{L^2(\partial\NN_k)}^2\geq\frac{\beta^2}4\qquad\forall k\in\widetilde I_n
\] 
as soon as $n$ is sufficiently large, but since $\widetilde I_n$ contains a total number of indices of order $\underline{i}_{n,\varepsilon}$, we then have
\[
\mu=\|u_n\|_{L^2(\Q)}^2\geq\sum_{k\in\widetilde I_n}\|u_n\|_{L^2(\partial\NN_k)}^2\geq \frac{\underline{i}_{n,\varepsilon}}{48}\beta^2\to+\infty\,,
\]
providing the contradiction we seek. This shows that there exists a sequence $(k_n)_n$ such that $k_n\to+\infty$ and $\|u_n\|_{L^\infty(\partial\NN_{k_n})}\to0$ as $n\to+\infty$.

Recall that all the previous discussion applies to any fixed $\varepsilon>0$. Take now a sequence $(\varepsilon_m)_m$ such that $\varepsilon_m\to0^+$ as $m\to+\infty$. For every fixed $m$ we know that $\displaystyle \lim_n\underline{i}_{n,\varepsilon_m}=+\infty$, so there exists $N_m\in\N$ such that $\underline{i}_{n,\varepsilon_m}\geq m$ for every $n\geq N_m$ and every $m\in\N$. Furthermore, since, for every fixed $m$, there exists a sequence  $(k_n^m)_n$ such that $k_n^m\to+\infty$ and $\|u_n\|_{L^\infty(\partial\NN_{k_n})}\to0$ as $n\to+\infty$, there exists $k_m\geq \underline{i}_{N_m,\varepsilon_m}\geq m$ and $n_m\geq N_m$ such that $\|u_{n_m}\|_{L^\infty(\partial\NN_{k_m})}\leq\varepsilon_m$ for every $m$. Hence, renaming $u_m=u_{n_m}$, we have 
\begin{equation}
\label{eq:um}
\lim_{m\to+\infty}\underline{i}_{N_m,\varepsilon_m}=+\infty,\qquad \|u_{m}\|_{L^\infty(\partial\NN_{k_m})}\leq\varepsilon_m\quad\forall m,
\end{equation}
for suitable $k_m\geq \underline{i}_{N_m,\varepsilon_m}$. Moreover, by definition of $\NN_{k_m}$, we also have $\|u_m\|_{L^\infty(\D\setminus\NN_{k_m})}\leq\varepsilon_m$ for every $m$. 

\begin{figure}[t]
	\centering
	\begin{tikzpicture}[scale=.6,>=latex]
		\draw[step=1] (9,1) grid (13,5);
		\foreach \i in {9,...,13}
		\foreach \j in {1,...,5}
		{\node at (\i,\j) [nodo] {};}
		
		\draw[step=1] (8,2) grid (12,6);
		\foreach \i in {8,...,12}
		\foreach \j in {2,...,6}
		{\node at (\i,\j) [nodo] {};}
		
		\draw[step=1] (7,3) grid (11,7);
		\foreach \i in {7,...,11}
		\foreach \j in {3,...,7}
		{\node at (\i,\j) [nodo] {};}
		
		\draw[step=1] (6,4) grid (10,8);
		\foreach \i in {6,...,10}
		\foreach \j in {4,...,8}
		{\node at (\i,\j) [nodo] {};}
		
		\draw[step=1] (5,5) grid (9,9);
		\foreach \i in {5,...,9}
		\foreach \j in {5,...,9}
		{\node at (\i,\j) [nodo] {};}
		
		\draw[step=1] (4,6) grid (8,10);
		\foreach \i in {4,...,8}
		\foreach \j in {6,...,10}
		{\node at (\i,\j) [nodo] {};}

		\draw[step=1] (8,4) grid (10,6);
		\foreach \i in {8,...,10}
		\foreach \j in {4,...,6}
		{\node at (\i,\j) [nodo] {};}
		\draw (3,7) grid (7,11);
		\foreach \i in {3,...,7}
		\foreach \j in {7,...,11}
		{\node at (\i,\j) [nodo] {};}
		\draw[dotted,step=1] (1,-1) grid (13,13);
		\foreach \i in {1,...,13}{
			\draw[dotted] (\i,13)--(\i,14);
			\draw[dotted] (\i,-2)--(\i,-1);}
		\foreach \i in {-1,...,13}{
			\draw[dotted] (0,\i)--(1,\i);
		}
	\end{tikzpicture}
	\caption{The embedding of $\NN_{k_m}$ into $\HH$.}
	\label{fig:Nrn}
\end{figure}
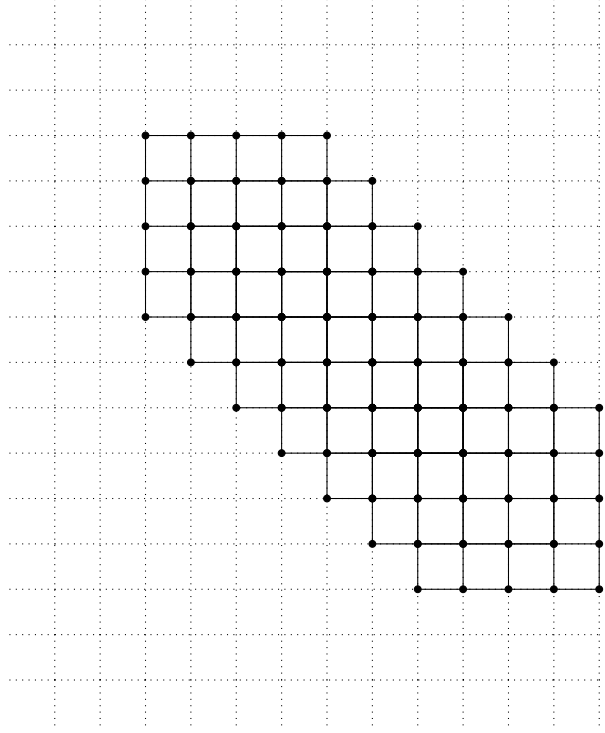

Denote now by $u_{m,1}$ the restriction of $u_m$ to $\NN_{k_m}$. Since by assumption $\displaystyle\|u_m\|_{L^\infty(\D\cap\NN_{k_m})}\to c$ as $m\to+\infty$, setting $v_{m,1}:=(u_{m,1}-\varepsilon_m)_+$ we have $v_{m,1}\not\equiv0$ on $\NN_{k_m}$, $v_{m,1}\equiv0$ on $\partial\NN_{k_m}$, and $\displaystyle \liminf_{m}E(v_{m,1},\NN_{k_m})\leq \liminf_m E(u_{m,1},\NN_{k_m})$.  Furthermore, since the $L^\infty$ norm of $v_{m,1}$ is bounded away from zero uniformly on $m$ and $(v_{m,1})_m$ is bounded in $H^1(\NN_{k_m})$, we also have $\displaystyle \liminf_m\|v_{m,1}\|_{L^2(\NN_{k_m})}>0$. Note that, by construction, any function supported on $\NN_{k_m}$ that vanishes on $\partial\NN_{k_m}$ can be seen as a function defined on the whole half-grid $\HH$, exploiting the natural embedding of $\NN_{k_m}$ in $\HH$ (see e.g. Figure \ref{fig:Nrn}) and extending the function to zero on $\HH\setminus\NN_{k_m}$. 

Let then $u_{m,2}$ be the restriction of $u_m$ to $\Q\setminus\NN_{k_m}$. If $\|u_{m}\|_{L^\infty(\Q\setminus\NN_{k_m})}\to0$ as $m\to+\infty$, then $\displaystyle\liminf_m E(u_{m,2}, \Q\setminus\NN_{k_m})\geq0$, and we obtain
\[
\begin{split}
\E_\Q(\mu)=\lim_m E(u_m,\Q)=&\,\lim_m (E(u_{m,1}, \NN_{k_m})+E(u_{m,2},\Q\setminus\NN_{k_m}))\geq \liminf_m E(v_{m,1}, \NN_{k_m})\\
=&\,\liminf_m E(v_{m,1}, \HH)\geq\liminf_m\E_\HH\left(\|v_{m,1}\|_{L^2(\HH)}^2\right)\geq \E_\HH(\mu)\,,
\end{split}
\]
where we used the fact that $\|v_{m,1}\|_{L^2(\HH)}^2\leq\mu$ for every $m$ and the monotonicity of $\E_\HH(\cdot)$. Since the above chain of inequalities is incompatible with the hypothesis $\E_\Q(\mu)<\E_\HH(\mu)$, we must have $\|u_{m}\|_{L^\infty(\Q\setminus\NN_{k_m})}\geq \beta>0$ for some $\beta$ independent of $m$. Hence, setting $v_{m,2}:=(u_{m,2}-\varepsilon_m)_+$, we have that $v_{m,2}\not\equiv0$ on $\Q\setminus\NN_{k_m}$, $\displaystyle \liminf_m\|v_{m,2}\|_{L^2(\Q\setminus\NN_{k_m})}>0$, and $\displaystyle \liminf_{m}E(v_{m,2},\Q\setminus\NN_{k_m})\leq \liminf_mE(u_{m,2},\Q\setminus\NN_{k_m})$. Moreover, by construction $\partial\left(\Q\setminus\NN_{k_m}\right)\setminus\partial\Q$ is given by all the vertices of $\partial\NN_{k_m}$, and $v_{m,2}\equiv0$ on this set. Hence, arguing exactly as we did in Case (i), we can rotate the portion of $\Q\setminus\NN_{k_m}$ below the diagonal containing the vertices of $\D$   counterclockwise by 90 degrees about the vertex $\w_0$, to obtain a new graph $\widetilde\Q_{\text{\normalfont open}}$ that can be naturally embedded in $\HH$. The graph $\widetilde\Q_{\text{\normalfont open}}$ is almost identical to $\Qo$ defined in Case (i), the only differences being in the profile of the boundary in correspondence of the region where $\NN_{k_m}$ has been removed from $\Q$. Since any function on $\Q\setminus\NN_{k_m}$ that vanishes on $\partial\left(\Q\setminus\NN_{k_m}\right)\setminus\partial\Q$ can be seen as a function on $\widetilde\Q_{\text{\normalfont open}}$, we can then think of $v_{m,2}$ as a function on $\widetilde\Q_{\text{\normalfont open}}$, and thus on $\HH$ extending it to zero on $\HH\setminus\widetilde\Q_{\text{\normalfont open}}$. All in all, we have
\[
\begin{split}
\E_\Q(\mu)=\lim_m E(u_m,\Q)=&\,\lim_m\left(E(u_{m,1}, \NN_{k_m})+E(u_{m,2},  \Q\setminus\NN_{k_m})\right)\\
\geq&\,\liminf_mE(v_{m,1},\HH)+\liminf_mE(v_{m,2},\HH)\\
\geq&\,\liminf_m\E_\HH\left(\|v_{m,1}\|_{L^2(\HH)}^2\right)+\liminf_m\E_\HH\left(\|v_{m,2}\|_{L^2(\HH)}^2\right)\geq\E_\HH(\mu)\,,
\end{split}
\]
where we used the fact that $\liminf_m\|v_{m,1}\|_{L^2(\HH)}>0$, $\liminf_m\|v_{m,2}\|_{L^2(\HH)}>0$, $\|v_{m,1}\|_{L^2(\HH)}^2+\|v_{m,2}\|_{L^2(\HH)}^2\leq\mu$ for every $m$, and the subadditivity of $\E_\HH$. As we found again a contradiction, we conclude.
\end{proof}

\begin{proof}[Conclusion of the proof of Theorem \ref{thm:half-quart} for $X=\Q$]
	The proof follows exactly the same lines as the proof for $\HH$ presented in Section \ref{subsec-halfgrid}, by simply replacing $\Gamma$ with $\HH$ and $\HH$ with $\Q$ and using the existence criterion provided by Lemma \ref{lem-critQ} in place of that of Lemma \ref{lem-critH}. 
\end{proof}


\subsection{A modified quarter-grid.}
\label{subsec-further}

In the next sections we will need to consider a slight modification of the quarter-grid $\Q$, obtained by adding to it infinitely many pendants of unit length: the modified quarter-grid $\Qm$ depicted in Figure {\sc \ref{fig:modquartgrid_b}}. The proof of Theorem \ref{thm:half-quart} can be adapted, with only minor modifications, to $\Qm$. Since we will not need the full resulting statement, we do not reproduce the argument here and restrict ourselves to the additional small-mass estimate required below.

\begin{figure}[t]
\centering
\begin{tikzpicture}[scale=1,>=latex]
\draw[step=1] (0,0) grid (4,4);
\foreach \j in {0,...,4}{
\draw[dashed] (-.5,\j)--(0,\j);}
\foreach \i in {0,...,4}{
\draw[dashed] (\i,4)--(\i,4.5);}
\foreach \i in {0,...,5}
\foreach \j in {0,...,4}
{\node at (\i,\j) [nodo] {};}
\foreach \j in {0,...,4}{
\draw (4,\j)--(5,\j);}
\node at (2,-.5) [infinito] {};
\end{tikzpicture}
\caption{Modified quarter-grid $\Qm$.}
\label{fig:modquartgrid_b}
\end{figure}
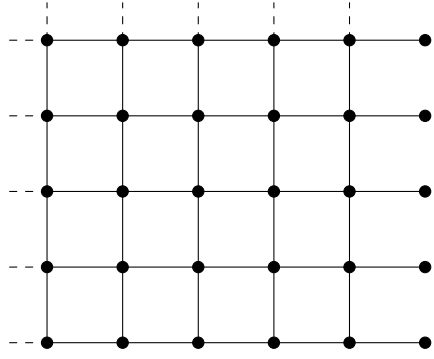

\begin{proposition}
	\label{prop:2dQm}
Let $p\in(2,6)$. Then, as $\mu\to0$, it holds
\begin{equation}
	\label{eq:QGas}
\begin{cases}
-C_p^{-1}\mu^{\frac{2}{4-p}}\leq\E_{\Qm}(\mu)\leq-C_p\mu^{\frac{2}{4-p}} & \text{if }p\in(2,4)\\[.2cm]
\E_{\Qm}(\mu) =0                                  & \text{if }p\in[4,6),
\end{cases}
\end{equation}
where $C_p>0$ is a constant depending only on $p$.
\end{proposition}

\begin{proof}
Since the second line of \eqref{eq:QGas} is a mere consequence of \eqref{eq-energylevelHQ} with $X=\Qm$, we are left to prove the first line.

Let then $p\in(2,4)$ be fixed. Observe first that any compactly supported function on the half-grid $\HH$ can be thought of as a function defined on $\Qm$, simply setting its support sufficiently far from the rightmost boundary of $\Qm$. Hence, by density of $C_c^\infty(\HH)$ in $H^1(\HH)$, we immediately have
\begin{equation}
		\label{eq:K1}
		\E_{\Qm}(\mu)\leq \E_\HH(\mu)\lesssim_p-\mu^{\frac{2}{4-p}}\qquad\text{as }\mu\to0\,,
\end{equation}
where the second inequality follows from \eqref{eq-energyrelationH} and \cite[Proof of Proposition 5.5]{Dovetta24} (with $d=2$).

To obtain \eqref{eq:QGas}, it thus remains to show that
\begin{equation}
	\label{eq:K2}
  \E_{\Qm}(\mu)\gtrsim_{\,p} - \mu^{\frac{2}{4-p}},\qquad \text{as}\quad\mu\to0.
\end{equation}
Preliminarily, note that, if $u$ is a normalized ground state with mass $\mu$ on $\Qm$, then, using \eqref{eq:K1} and arguing again as in the proof of \cite[Proposition 5.5]{Dovetta24},
\begin{equation}\label{eq:asymp}
  \frac1C\mu^{\frac{2}{4-p}}\leq  \| u'\|^2_{L^2(\Qm)}, \| u\|^p_{L^p(\Qm)} \leq C\mu^{\frac{2}{4-p}}\qquad \text{as}\quad\mu\to0,
\end{equation}
for some $C>0$ depending only on $p$.  
Moreover, if $w$ is a normalized ground state with mass $\mu$ on $\HH$, \cite[Proposition 5.5]{Dovetta24} gives also
\begin{equation}\label{eq:asympH}
 \frac1C\mu^{\frac{2}{4-p}}\leq	\| w'\|^2_{L^2(\HH)}, \| w\|^p_{L^p(\HH)} \leq C\mu^{\frac{2}{4-p}}\qquad \text{as}\quad\mu\to0.
\end{equation}

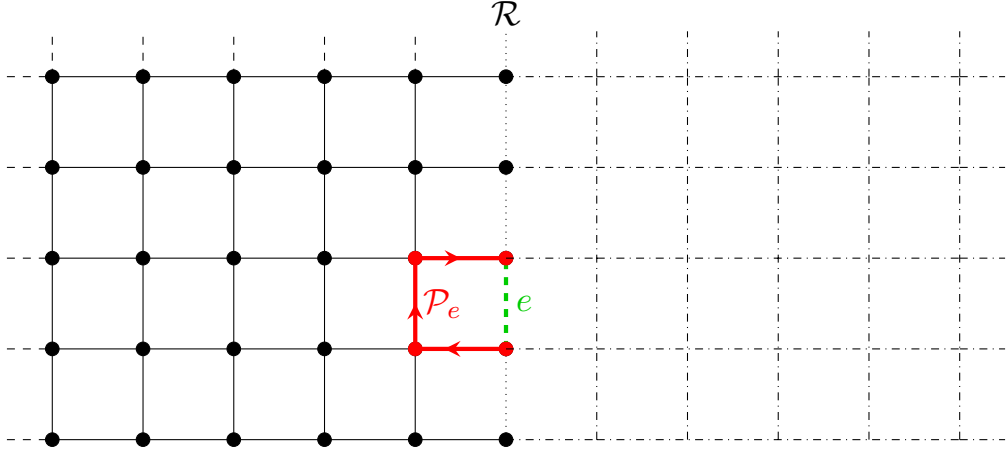
\begin{figure}[!ht]
\centering
\begin{tikzpicture}[scale=1.2,>=latex]
\draw[step=1] (0,0) grid (4,4);
\foreach \j in {0,...,4}{
\draw[dashed] (-.5,\j)--(0,\j);}
\foreach \i in {0,...,4}{
\draw[dashed] (\i,4)--(\i,4.5);}
\foreach \i in {0,...,5}
\foreach \j in {0,...,4}
{\node at (\i,\j) [nodo] {};}
\foreach \j in {0,...,4}{
\draw (4,\j)--(5,\j);}
\node at (2,-.5) [infinito] {};
\draw[dotted] (5,0)--(5,1);
\draw[dotted] (5,2)--(5,4.5);
\draw[midarrowtwo,ultra thick,red] (4,1)--(5,1);
\draw[midarrow,ultra thick,red] (4,2)--(5,2);
\draw[midarrow,ultra thick,red] (4,1)--(4,2);
\draw[dashed,ultra thick,green!80!black] (5,1)--(5,2);
\node[red] at (4,1) [nodo] {};
\node[red] at (5,1) [nodo] {};
\node[red] at (4,2) [nodo] {};
\node[red] at (5,2) [nodo] {};
\node at (5.2,1.5) [infinito] {${\color{green!80!black} e}$};
\node at (4.3,1.5) [infinito] {${\color{red} \mathcal{P}_e}$};
\node at (5,4.7) [infinito] {$\mathcal{R}$};
\draw[dash dot,step=1] (6,0) grid (10,4);
\foreach \j in {0,...,4}{
\draw[dash dot] (5,\j)--(6,\j);
\draw[dash dot] (10,\j)--(10.5,\j);}
\foreach \i in {6,...,10}{
\draw[dash dot] (\i,4)--(\i,4.5);}
	\end{tikzpicture}
	\caption{The embeddings of $\Qm$ into $\Q$ (the added edges that form the half-line $\mathcal{R}$ and complete $\Q$ are dotted) and then into $\HH$ (the added edges that complete $\HH$ are dash-dotted).}
	\label{fig:missingedge}
\end{figure}

Now, let $u \in H_{\mu}^1(\Qm)$ be a normalized ground state with mass $\mu$. Denote also by $v$ the function in $H^1(\Q)$ obtained by extending $u$ to $\Q$ in the following way (see Figure \ref{fig:missingedge}):
\begin{itemize}
	\item $\Qm$ is embedded into $\Q$ by adding infinitely many edges of unit length, each of which connects the endpoints of two consecutive pendants,
	\item on any such added edges $e$, that can be identified as the interval $[0,1]$, we set $v_e(x):=u_{|\mathcal{P}_e}(3x)$, with $\mathcal{P}_e$ denoting the path consisting of the three edges of $\Qm$ that form a square cell with $e$ (see Figure \ref{fig:missingedge} for an example);
	\item $v_{\Qm}\equiv u$.
\end{itemize}
Subsequently, denote by $w$ the function in $H^1(\HH)$ which coincides with $v$ on $\Q$ and on $\HH\setminus\Q$ is the reflection of $u$ with respect to the vertical half-line $\mathcal{R}$ depicted in Figure \ref{fig:missingedge}. Thus, by construction
\begin{equation}
    \int_\HH w^2\dx=2\int_{\Qm} u^2 \dx+ \sum_{e\in\mathcal{R}}\int_{e} v^2\dx=2\int_{\Qm} u^2 \dx+ \frac{1}{3}\sum_{e\in\mathcal{R}} \int_{\mathcal{P}_e} u^2 \dx\le \frac{8}{3} \mu.
\end{equation}
Furthermore,
\begin{align}
    \int_\HH |w'|^2\dd x= & \, 2\int_{\Qm} |u'|^2 \dx+ 3\sum_{e\in\mathcal{R}}\int_{\mathcal{P}_e}|u'|^2\dx,\\[.2cm]
     \int_\HH |w|^p\dd x = & \, 2\int_{\Qm} |u|^p\dx + \frac{1}{3}\sum_{e\in\mathcal{R}} \int_{\mathcal{P}_e} |u|^p\dx.
\end{align}
Thus, from \eqref{eq:asymp},
\begin{align*}
    E(w,\HH)&= 2E(u,\Qm)+\frac{3}{2}\sum_{e\in\mathcal{R}} \int_{\mathcal{P}_e} |u'|^2\dx-\frac{1}{3p}\sum_{e\in\mathcal{R}}\int_{\mathcal{P}_e}|u|^p\dx\\[.2cm]
    &\le 2E(u,\Qm)+\frac{3}{2}\sum_{e\in\mathcal{R}} \int_{\mathcal{P}_e} |u'|^2\dx\le 2E(u, \Qm) +3 \| u'\|^2_{L^2(\Qm)}\\[.2cm]
    &\,\,\quad\qquad\qquad\qquad\qquad\qquad\qquad\qquad\le 2 E(u, \Qm)+K\mu^{\frac{2}{4-p}},
\end{align*}
with $K>0$ a constant only depending on $p$. As a consequence, using \eqref{eq:asympH} and denoting by $\varphi$ a normalized ground state with mass $\frac83\mu$ on $\HH$, for small $\mu$
\begin{multline}
\label{eq:kmu}
    -K' \cdot \mu^{\frac{2}{4-p}}\leq-\frac1p\|\varphi\|_{L^p(\HH)}^p\leq E(\varphi,\HH)= \E_\HH\left(\frac{8}{3}\mu\right)\\[.2cm]
    \le \E_\HH\left(\int_\HH w^2\dx\right)\le E(w, \HH)\le 2E(u,\Qm)+K\mu^{\frac{2}{4-p}},
\end{multline}
with $K'>0$ another constant only depending on $p$. Hence,
\begin{equation}
    \E_{\Qm}(\mu)=E(u,\Qm)\ge -K_2\cdot\mu^{\frac{2}{4-p}},\qquad \text{as}\quad\mu\to0,
\end{equation}
for some constant $K_2>0$ (depending on $p$), which concludes the proof.
\end{proof}


\section{Nonexistence results: proof of Theorems \ref{thm:noexsmall} and \ref{thm:noexlarge}}
\label{sec-nonexistence}

In this section we prove our nonexistence results on the defected grid $\G$ depicted in Figure \ref{fig:half+lad_1}.


\subsection{Nonexistence for small mass}
\label{subsec-noexsmall}
First, we focus on the nonexistence for small mass, which is claimed in Theorem \ref{thm:noexsmall}.

\begin{proof}[Proof of Theorem \ref{thm:noexsmall}]

As a preliminary step, we prove that
\begin{equation}\label{eq:GLineq}
    \E_{\G}(\mu)\le\E_{\L_1}(\mu),\qquad\forall \mu>0,
\end{equation}
with $\L_1$ the ladder graph with edges of length 1 introduced in Section \ref{sec-preliminaries}. To see this, first we note that $\L_1$ can be regarded as embedded into $\G$. Moreover, recalling \eqref{eq-ladderconstr},
\[
\L_1=\bigg(\bigcup_{k\leq -n}\mathcal \cell_1^k\bigg)\cup \bigg(\bigcup_{k>-n}\mathcal \cell_1^k\bigg)=:\mathcal L^n_{-}\cup \mathcal L^n_{+},\qquad \forall n\in\N.
\]
As a consequence, letting $\varphi$ be a ground state with mass $\mu$ on $\L_1$, then
\[
\Vert \varphi\Vert_{L^\infty(\mathcal L^n_-)}\le C \Vert \varphi\Vert_{H^1(\mathcal L^n_-)}=o(1)\,,\qquad \Vert \varphi\Vert_{L^2(\mathcal L^n_-)}=o(1),\qquad \text{as}\quad n\to\infty.
\]
Now, consider a sequence of smooth cutoff functions $(\chi_n)_n$, such that $\chi_n\equiv 1$ on $\mathcal L^n_+$, $\supp(\chi_n)\subset \mathcal L^{n+1}_+$ and $\vert \chi'_n\vert \leq C$, for some constant $C>0$ not depending on $n$. Setting $v_n:=\chi_n \varphi$, one has
\[
E(v_n,\L_1)=\E_{\L_1}(\mu)+o(1),\qquad\text{as}\quad n\to\infty.
\]
Thus, if we take a sequence $(\kappa_n)_n\subset\R$ such that $\Vert \kappa_n v_n\Vert_{L^2(\L_1)}^2=\mu$ for every $n\in\N$, we have that $\kappa_n>1$ for every $n\in\N$ and that $\lim_{n}\kappa_n=1$, whence $(\kappa_nv_n)_n$ is a minimizing sequence for $\E_{\L_1}(\mu)$. In addition, if we denote by $(w_n)_n$ a sequence such that $w_n$ is a translation of $\kappa_nv_n$ with support in the half-ladder $\L_+^1$ for every $n\in\N$, then $(w_n)_n$ is a minimizing sequence for $\E_{\L_1}(\mu)$ too. On the other hand, by construction, the sequence $(w_n)_n$ can be seen as supported on $\G$ by a suitable extension to zero and, in particular, $(w_n)_n\subset H^1_\mu(\G)$. Therefore
\[
\E_\G(\mu)\leq E(w_n,\G)=E(w_n,\L_1)=\E_{\L_1}(\mu)+o(1),\qquad \text{as}\quad  n\to\infty,
\]
which proves \eqref{eq:GLineq}.
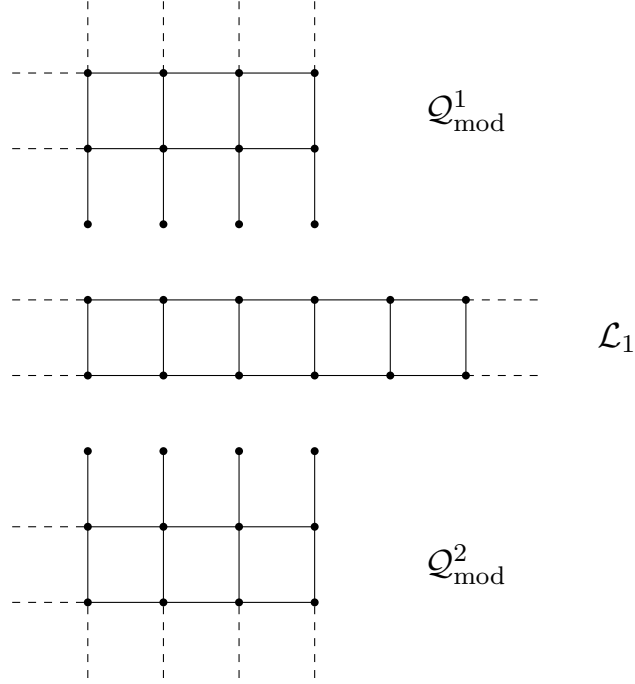
\begin{figure}[t]
	\centering
	\begin{tikzpicture}[scale=1,>=latex]
		\draw[step=1] (0,1) grid (3,2);
		\draw[step=1] (0,-2) grid (3,-1);
		\draw[step=1] (0,4) grid (3,5);
		\foreach \i in {0,...,3}{
			\draw[dashed] (\i,5)--(\i,6);}
		\foreach \i in {0,...,3}{
			\draw (\i,3)--(\i,4);}
		\foreach \i in {0,...,3}{
			\draw (\i,0)--(\i,-1);}
		\foreach \j in {1,2}{
			\draw[dashed] (-1,\j)--(0,\j);}
		\foreach \j in {4,5}{
			\draw[dashed] (-1,\j)--(0,\j);}
		\foreach \j in {-1,-2}{
			\draw[dashed] (-1,\j)--(0,\j);}
		\foreach \i in {0,...,3}{
			\draw[dashed] (\i,-3)--(\i,-2);}
		\foreach \i in {0,...,3}
		\foreach \j in {-2,...,5}
		{\fill (\i,\j) circle(1.5pt);}
		\draw (3,1)--(5,1);
		\draw (3,2)--(5,2);
		\draw (4,1)--(4,2);
		\draw (5,1)--(5,2);
		\draw[dashed] (5,1)--(6,1);
		\draw[dashed] (5,2)--(6,2);
		\fill (5,1) circle(1.5pt);
		\fill (5,2) circle(1.5pt);
		\fill (4,1) circle(1.5pt);
		\fill (4,2) circle(1.5pt);
		\node[scale=1.2] at (5,4.5) {$\Qm^1$};
		\node[scale=1.2] at (5,-1.5) {$\Qm^2$};
		\node[scale=1.2] at (7,1.5) {$\L_1$};
	\end{tikzpicture}
	\caption{The decomposition of the graph $\mathcal{G}$ in $\Qm^1$, $\Qm^2$ and $\L_1$.}
	\label{fig:half+lad_2}
\end{figure}
Now, assume by contradiction that there exists $u\in H_{\mu}^1(\G)$ such that $E(u,\G)=\E_\G(\mu)$. Without loss of generality, we may assume $u>0$ on $\G$. Therefore, divide $\G$ into two modified quarter-grids $\Qm^1$, $\Qm^2$ and a ladder graph $\L_1$ as depicted in Figure \ref{fig:half+lad_2} and denote by $u_1$, $u_2$, $u_3$ the restrictions of $u$ to $\Qm^1$, $\Qm^2$ and $\L_1$, respectively. Let then
\begin{equation}
\label{eq:m1m2}
m_1:=\int_{\Qm^1}|u_1|^2 \dx, \qquad  m_2:=\int_{\Qm^2}|u_2|^2 \dx,
\end{equation}
so that
\begin{equation}
    \int_{\L_1} u_3^2 \dx =\mu-(m_1+m_2)
\end{equation}
and
\begin{equation}
   E(u,\G)\ge \E_{\Qm^1}(m_1)+\E_{\Qm^1}(m_2)+\E_{\L_1}(\mu-(m_1+m_2)).
   \end{equation}
Since for $p\in(2,4)$ there holds
\[
\frac{2}{4-p}>\frac{p+2}{6-p}\,,
\]
by \eqref{eq:Las} and \eqref{eq:QGas} we obtain
\begin{gather*}
    \E_{\Qm^1}(m_1)\ge\E_{\L_1}(m_1)\\
    \E_{\Qm^2}(m_2)\ge\E_{\L_1}(m_2)
\end{gather*}
for $\mu>0$ small and for every $p\in(2,6)$. As a consequence, 
\begin{equation}
    E(u,\G)\ge \E_{\L_1}(m_1)+\E_{\L_1}(m_2)+\E_{\L_1}(\mu-(m_1+m_2))>\E_{\L_1}(\mu),
\end{equation}
the last inequality following from the fact that $m_1,\,m_2\neq0$, in view of the strict subadditivity of $\E_{\L_1}$. Since this contradicts \eqref{eq:GLineq}, we conclude.
\end{proof}


\subsection{Nonexistence for large mass}

Here we prove the nonexistence for large mass stated in Theorem \ref{thm:noexlarge}. To do this, we exploit the obvious fact that $\G$ contains infinitely many different copies of the ladder graph $\L_1$. Indeed, besides the horizontal copy of $\L_1$ used in Section \ref{subsec-noexsmall}, infinitely many (non-disjoint) vertical copies of $\L_1$ can be found on the left part of $\G$. Therefore, when we say below that a given subset $A\subset\G$ is not contained in any ladder graph with edges of length $1$, we mean that there is no copy of $\L_1$ in $\G$ that fully contains $A$.  

\begin{remark}
\label{rem:gsladder}
   Note that,  if $u$ is a ground state with mass $\mu$ on $\G$, then the set
    \[
    \left\{e \in \mathbf{E}_\G: E(u,e)<0\right\}
    \]
    is not contained in any ladder graph with edges of length 1 embedded into $\G$. Indeed, assume by contradiction that $u$ is a ground state with mass $\mu$ on $\G$ for which there exists a ladder graph $\L_1$ embedded into $\G$ such that $\displaystyle\left\{e \in \mathbf{E}_\G: E(u,e)<0\right\}\subseteq \L_1$.
    Since we may assume without loss of generality that $u>0$ on $\G$, this would imply $\E_\G(\mu)=E(u,\G)\geq E(u,\L_1)\geq\E_{\L_1}\left(\|u\|_{L^2(\L_1)}^2\right)>\E_{\L_1}(\mu)$, owing to the positivity of $u$ on $\G\setminus\L_1$ and to the strict monotonicity of $\E_{\L_1}(\cdot)$. However, this is impossible, as it contradicts \eqref{eq:GLineq}.
\end{remark}
The proof of Theorem \ref{thm:noexlarge} is based on the following proposition.
\begin{proposition}
\label{prop:bigmu}
    Let $(\mu_n)_n$ be such that $\mu_n \to +\infty$ as $n\to+\infty$, and let $u_n$ be a ground state with mass $\mu_n$ on $\G$, for every $n$. Then, for $n$ large enough there exists a unique edge $e_n\in\mathbf{E}_\G$ such that
    \[
    E(u_n,e_n)<0.
    \]
\end{proposition}

\begin{proof}
Observe first that the structure of $\G$ ensures that any nonnegative function in $H^1(\G)$ attains almost every value in its image at least twice on $\G$. Then, by a standard rearrangement argument (see, e.g., the proof of \cite[Proposition 4.3]{Adami19b}) one obtains
\begin{equation}
\label{eq:EG>ER}
 \E_{\G}(\mu)\ge \E_{\R}(\mu)\qquad\forall \mu>0.
\end{equation}
On the other hand, by \cite[Proposition 3.2]{Serra16}, there exists $C>0$ (depending only on $p$) such that
\begin{equation}
\label{eq-estimateabove}
\E_{\G}(\mu)\le\E_{\R}(\mu)+C\exp(-\mu^\beta/C),\qquad\text{as}\quad\mu\to+\infty
\end{equation}
($\beta$ as in \eqref{eq:ab}). Moreover, these estimates are unchanged (with the same $C$ in \eqref{eq-estimateabove}) even when the length of the edges in $\G$ is greater than 1.

Now, for every $n$, let $u_n$ be a normalized ground state with mass $\mu_n$  on $\G$. Arguing as in the proof of Theorem \ref{thm:ladder}, setting
\begin{equation}
\label{def:v}
  \G_n:=\mu_n^{\beta}\G\qquad\text{\and}\qquad v_n(x):=\mu_n^{-\alpha}u_n\left(\mu_n^{-\beta}x\right),\quad\forall x\in\G_n,
\end{equation}
with $\alpha,\,\beta$ as in \eqref{eq:ab}, it follows that
\begin{equation}
\label{prop:v}
v_n\in H_1^1(\G_n)\qquad\text{and}\qquad E(v_n,\G_n)=\frac{E(u_n,\G)}{\mu_n^{2\beta+1}}=\frac{\E_\G(\mu_n)}{\mu_n^{2\beta+1}}.
\end{equation}
Thus, combining all the previous observations, we have
\[
\frac{\E_{\R}(\mu_n)+C\exp(-\mu_n/C)}{\mu_n^{2\beta+1}}\ge E(v_n,\G_n)\ge \frac{\E_{\R}(\mu_n)}{\mu_n^{2\beta+1}},\qquad\text{as}\quad n\to+\infty,
\]
so that, since arguing again as in the proof of Proposition \ref{thm:ladder} and using the scale invariance of $\R$ it follows that $\E_{\R}(\mu_n)=\mu_n^{2\beta+1}\E_{\R}(1)$, one finds
\begin{equation}\label{eq:encomp}
\E_{\R}(1)+o(\mu_n^{-q})\ge E(v_n,\G_n)\ge\E_{\R}(1),\quad\forall q>0,\qquad\text{as}\quad n\to+\infty.
\end{equation}

In addition, \eqref{def:v} and \eqref{prop:v} imply that, for every $n$, $v_n$ is a normalized ground state with mass 1 on $\G_n$. Hence, it is well known that the $v_n\,$'s are positive (up to a change of sign) and that, for every $n$, there exists $\lambda_n>0$ such that
\begin{equation}
\label{eq:vneq}
\begin{cases}
    -v''_n+\lambda_n v_n=|v_n|^{p-2}v_n & \text{on every edge }e\in\mathbf{E}_{\G_n}\\
    \displaystyle {\sum_{e\succ \v}\frac{\dd {v_n}_{|e}}{\dd x_e}(\v)=0} & \forall \v\in\mathbf{V}_{\G_n}
\end{cases}
\end{equation}
(where $\sum_{e\succ \v}$ means that the sum is extended to all the edges $e\in\mathbf{E}_{\G_n}$ that are incident at $\v$ and $\frac{\dd {v_n}_{|e}}{\dd x_e}$ is a shorthand notation for the outgoing  derivative of ${v_n}_{|e}$ at $\v$). Also, setting
\begin{equation}
\label{eq-lambdatilde}
\widetilde{\lambda}_n:=\mu_n^{2\beta}\lambda_n,
\end{equation}
we obtain that
\begin{equation}
\label{eq:ELu}
-u''_n+\widetilde{\lambda}_n u_n=|u_n|^{p-2}u_n,\qquad \text{on}\quad\G\setminus\mathbf{V}_{\G}.
\end{equation}
As a consequence, by \eqref{eq-estimateabove},
\[
-K \mu_n^{2\beta+1}\ge\frac{1}{2}\E_{\R}(\mu_n)\ge E(u_n,\G)=\left(\frac{1}{2}-\frac{1}{p}\right)\int_{\G}|u'_n|^2\dx-\frac{\widetilde{\lambda}_n \mu_n}{p},\qquad\text{as}\quad n\to+\infty,
\]
for some constant $K>0$ (only depending on $p$). Hence, for $n$ large enough, $\widetilde{\lambda}_n \mu^{-2\beta}_n\gtrsim 1$, so that by \eqref{eq-lambdatilde}
\begin{equation}\label{eq:lgone}
\lambda_n \gtrsim 1.
\end{equation}
On the other hand, since \eqref{eq:encomp} implies
\begin{equation}
	\label{eq-enGNnegative}
E(v_n,\G_n)\le\frac{1}{2}\E_{\R}(1)<0,
\end{equation}
combining with \eqref{eq-gn1d} we have
\[
\Vert v'_n\Vert _{L^2(\G_n)}^2\lesssim \Vert v_n\Vert^p_{L^p(\G_n)} \lesssim \Vert v'_n\Vert_{L^2(\G_n)}^{\frac{p}{2}-1}.
\]
As $p\in(2,6)$, this implies that $\big(\Vert v'_n\Vert_{L^2(\G_n)}\big)_n$, and thus $\big(\Vert v_n\Vert_{L^\infty(\G_n)}\big)_n$, are uniformly bounded sequences. Thus, since $\lambda_n\le \Vert v_n\Vert^{p-2}_{L^\infty(\G_n)}$ by \eqref{eq:vneq}, $(\lambda_n)_n$ is uniformly bounded too and, up to subsequences, $\lambda_n \to \lambda>0$.

Now, for every $n$, define
\[
 N_n:=\#\left\{x\in\G_n:v_n(x)=\max_{y\in B}v_n(y),\:\text{for some open}\:B\subset\G_n\right\}.
\]
Let us show that
\begin{equation}
 \label{eq-nummax}
 \sup_n N_n<+\infty\,.
\end{equation}
Fix an arbitrary edge $e$ in $\G_n$ and recall that it can be identified with the interval $[0,\mu_n^\beta]$. Then, on such interval, multiplying the first line in \eqref{eq:vneq} by $v_n'$ and integrating over $[0,x]$, one gets
\begin{equation}\label{eq:cn}
\frac{1}{2}\vert v'_n(x)\vert^2+\frac{1}{p}\vert v_n(x)\vert^p-\frac{\lambda_n}{2}\vert v_n(x)\vert^2=c_{n,e},\qquad\forall x\in[0,\mu_n^\beta],
\end{equation}
where the left-hand side is the mechanical energy associated with $v_n$ evaluated at the point $x$ of the edge $e$ (we omit to repeat the restriction to $e$ for the sake of simplicity) and $c_{n,e}$ is a constant depending only on $n$ and $e$. However, if $e$ contains a local maximizer $x_n$ for $v_n$, then $v'_n(x_n)=0$ (even at a vertex). As a consequence, if $e$ contains a local maximizer for $v_n$, then by the uniform boundedness of $\lambda_n$ and $\Vert v_n\Vert_{L^\infty(\G_n)}$ we have that $c_{n,e}$ is uniformly bounded with respect to $n$, as well as $\Vert v'_n\Vert_{L^\infty(e)}$. Also, evaluating the first line of \eqref{eq:vneq} at a local maximizer $x_n$, (again, even at a vertex) one finds
\[
\lambda_n v_n(x_n)=v''_n(x_n)+v^{p-1}_n(x_n)\leq v^{p-1}_n(x_n),
\]
and thus by \eqref{eq:lgone}
\[
 v_n(x_n)\geq \lambda^{1/(p-2)}_n\geq C,
\]
with $C>0$ a constant depending only on $p$. As a consequence, given the uniform bound on $v'_n$ on the edges where $v_n$ has a local maximizer, we deduce that there exists $\ell>0$, independent of $n$, such that
\begin{equation}
	\label{eq-nondegeneracy}
v_n(x)\ge \frac{C}{2},\qquad\forall x\in B_{\ell}(x_n,\G_n),\quad\forall x_n\:\text{ local maximizer of }\:v_n,
\end{equation}
where $B_{\ell}(x_n,\G_n)$ denotes the open ball of radius $\ell$ centered at $x_n$ on $\G_n$. Since $\lambda_n\to\lambda>0$ a classical phase portrait argument shows that the distance between two local maximizers of $v_n$ is bounded away from zero as $n\to+\infty$. Hence, one can choose $\ell$ so that  $B_{\ell}(x_n,\G_n)\cap B_{\ell}(y_n,\G_n)=\emptyset$ for every pair $x_n,\,y_n$ of local maximizers of $v_n$ and for every $n$, so that
\[
1=\Vert v_n\Vert^2_{L^2(\G_n)}\ge N_n \frac{C^2\ell}{4}\,,\qquad\text{and thus}\qquad N_n\leq\frac{4}{C^2\ell}\,,
\]
which proves \eqref{eq-nummax}.

As a further step, for every $n$, define
\begin{equation}
    \alpha_n:=\max_{e\in\mathbf{E}_{\G_n}}\min_{x\in e} v_n(x).
\end{equation}
For $n\to +\infty$, we immediately have that $\alpha_n\to0$, since $1=\Vert v_n\Vert^2_{L^2(\G_n)}\gtrsim\mu_n^{\beta}\alpha_n^2$. Moreover, for any $n$, set
\begin{equation}
\label{eq:An}
    A_n:=\{v_n>\alpha_n\}.
\end{equation}
First, one notes that each $A_n$ is the disjoint union of a finite number of connected components $A_{n,i}$, with $i\in I_n \subset \N$, and that $\sup_n\#I_n<+\infty$. Indeed, by definition, $v_n=\alpha_n$ on $\partial A_{n,i}$, $v_n>\alpha_n$ in $A_{n,i}$ and $v_n$ is continuous. Thus, each connected component contains at least one local maximizer of $v_n$ and one concludes by \eqref{eq-nummax}. Moreover, if a component $A_{n,i}$ contains more than one local maximizer of $v_n$, then it must contain at least one vertex of $\G_n$. Indeed, if this is not the case, then the edge containing $A_{n,i}$ should contain two local minimizers at different levels, which is excluded by a classical phase portrait analysis. With the very same argument, one can prove that any connected component of $A_n$ can contain at most one vertex of $\G_n$.

On the other hand, for $n$ large, 
\begin{equation}
\label{eq:Evg0}
    E(v_n,e)>0,\qquad \text{for every edge $e$ contained in $\G\setminus A_n$}.
\end{equation}
In order to see this, first observe that $v_n\leq\alpha_n$ on every edge $e$ contained in $\G\setminus A_n$. Thus, since $\displaystyle\lim_n \alpha_n=0$ and the length of the edges of $\G_n$ diverges as $n\to\infty$, one obtains
\begin{equation}
\label{eq:derto0}
v'_n(\mathbf{v}_1),\,v'_n(\mathbf{v}_2)\to 0\qquad \text{with $\mathbf{v}_1,\,\mathbf{v}_2$ the vertices of $e$}.
\end{equation}
Indeed, if on the contrary for some $i=1,\,2$
\[
|v'_n(\mathbf{v}_i)|\geq \delta >0 \qquad \text{for every $n$}
\]
(up to subsequences), then $v_n$ would be a solution of the first line \eqref{eq:vneq} on $e$ with initial conditions $v_n(\v_i)\to0$, $|v_n'(\v_i)|\geq\delta$. A direct inspection of the phase-plane would then show that $v_n$ must change sign on $e$, which is impossible since $v_n>0$ on $\G$ by assumption. 

By \eqref{eq:derto0}, $\displaystyle \lim_n\alpha_n=0$ and \cite[Lemma 4.5]{Dovetta22}, we obtain \eqref{eq:Evg0} for $n$ large. Furthermore, using \eqref{eq:Evg0} arguing in the same way one actually obtains
\begin{equation}
	\label{eq:Efuori}
	E(v_n,\G_n\setminus A_n)>0.
\end{equation}
We then estimate the asymptotic behavior of $E(v_n, A_{n,i})$ and $\|v_n\|_{L^2(A_{n,i})}^2$, for every $i\in I_n$. To this aim, for every $n$, denote by $B_n$ the connected component of $A_n$ such that
\begin{equation}
E(v_n,B_n)=\min_{i\in I_n} E(v_n, A_{n,i}).
\end{equation}
By \eqref{eq-nondegeneracy} and the definition of $A_n$, $|B_n|\geq \ell>0$. We treat separately the two cases where $B_n$ does or does not contain a vertex of $\G_n$. For the sake of simplicity, from now on it is understood that $n$ is large and the limits are taken for $n\to+\infty$. 

\smallskip
\emph{Case (i): $B_n$ does not contain any vertex of $\G_n$}. In this case $B_n$ is an open interval $(a_n,b_n)$ completely contained in some edge $\widetilde{e}_n$ of $\G_n$. With the same argument as before, one can check that the mechanical energy of $v_n$ tends to zero on $\widetilde{e}_n$, and thus on $B_n$. As a consequence, for analogous reasons, $v_n$ satisfies
\[
\begin{cases}
    -v''_n+\lambda_nv_n=v_n^{p-1} & \text{on } B_n,\\[.2cm]
    v_n(a_n),\,v_n(b_n) \to 0,  & \\[.2cm]
    v'_n(a_n),\,v_n'(b_n) \to 0. &
\end{cases}
\]
Since $\lambda_n\to\lambda>0$, classical arguments show that
\begin{equation}
	\label{eq-convsoliton}
    E(v_n,B_n) \longrightarrow E(\phi_\lambda,\R)\,,\qquad \int_{B_n} v_n^2\dx \longrightarrow \int_\R \phi_\lambda^2 \dx,
\end{equation}
where $\phi_\lambda$ is the so-called \emph{soliton} of the real line, i.e. the unique positive and even $L^2(\R)$-function on the real line that satisfies
\[
-\phi''_\lambda+\lambda\phi_\lambda=\phi_\lambda^{p-1} \qquad \text{on}\quad \R.
\]
For the sake of completeness, let us briefly sketch the strategy to establish these convergences. First, one can check that, denoting by $(x_n)_n$ a sequence such that $x_n$ is the local maximizer of $v_n$ on $B_n$ for every $n$, it follows that
\[
v_n(x_n)\to\left(\frac{p\lambda}{2}\right)^{\frac{1}{p-2}}=\phi_\lambda(0).
\]
As a consequence, defining the sequence $(w_n)_n$ such that, for every $n$, $w_n:J_n:=(a_n-x_n,b_n-x_n)\to\R$ is the function $w_n:=v_n(\cdot+x_n)$, one has that $w_n(0)\to\phi_\lambda(0)$, $w_n'(0)=0$,
\[
 -w''_n+\lambda_nw_n=w_n^{p-1}\qquad\text{on}\quad J_n
\]
and $w_n$ is uniformly bounded in $C^2(K)$ on every compact set $K\subset J_n$. Hence, by Ascoli-Arzel\`a, $w_n\to w$ (up to subsequences) in $C^1(K)$, on any compact set $K\subset J_n$, for every $n$ and using the equation such convergence can be enhanced to $C^2$. Now, since one can prove, e.g. by a contradiction argument, that $a_n-x_n\to-\infty$ and $b_n-x_n\to+\infty$, in the end we have that $w_n\to w$ in $C_{\mathrm{loc}}^2(\R)$, up to subsequences, with $w$ a nonnegative function satisfying
\begin{equation}
	\label{eq:w}
\begin{cases}
	-w''+\lambda w=w^{p-1} & \text{on } \R,\\[.2cm]
	w(0)=\phi_\lambda(0),  & \\[.2cm]
	w'(0)=0, &
\end{cases}
\end{equation}
namely $w\equiv\phi_\lambda$ by uniqueness of the solution to the Cauchy problem \eqref{eq:w}. Uniqueness also ensures that the limit along the sequence $(w_n)_n$ holds without passing to a subsequence. Finally, standard one-dimensional phase-plane estimates, together with the above convergence, yield
\[
\|v_n\|_{H^1(B_n)}=\|w_n\|_{H^1(J_n)}\to\|\phi_\lambda\|_{H^1(\R)},\quad\|v_n\|_{L^q(B_n)}=\|w_n\|_{L^q(J_n)}\to\|\phi_\lambda\|_{L^q(\R)},\quad q\in[2,+\infty),
\]
whence \eqref{eq-convsoliton}. Thus, by the definition of $B_n$, 
\begin{equation}
\label{eq:estimate}
E(v_n,A_{n,i})\ge E(\phi_\lambda,\R)+o(1)=\E_\R\big(\|\phi_\lambda\|_{L^2(\R)}^2\big)+o(1), \qquad \text{for every}\quad i \in I_n
\end{equation}
and
\begin{equation}
	\label{eq:massnozero}
	\|v_n\|_{L^2(A_{n,i})}^2\ge c>0, \qquad \text{for every}\quad i \in I_n.
\end{equation}

\smallskip
\emph{Case (ii): $B_n$ contains a vertex $\mathbf{v}_n$ of $\G_n$}. If $v_n(\mathbf{v}_n)\to0$, we restart from \eqref{eq:An} with $v_n(\mathbf{v}_n)$ in place of $\alpha_n$ and we fall again in the situation of case (i). If on the contrary $v_n(\mathbf{v}_n)\ge \delta>0$, then, arguing in an analogous way as before we have
\begin{equation}
	\label{eq-convstargraph}
	E(v_n,B_n) \longrightarrow E(\psi_\lambda,\S)\,,\qquad\int_{B_n} v_n^2\dx \longrightarrow \int_\S \psi_\lambda^2 \dx,
\end{equation}
where now $\psi_\lambda$ is a nonnegative $H^1$-function on an infinite star-graph $\S$ with vertex $\mathbf{v}$, consisting of three or four half-lines, that satisfies
\[
	\begin{cases}
		-\psi_\lambda''+\lambda \psi_\lambda=\psi_\lambda^{p-1} & \text{on every edge of }\S\\
		\displaystyle {\sum_{e\succ \v}\frac{\dd {\psi_\lambda}_{|e}}{\dd x_e}(\v)=0}\,. &
	\end{cases}
\]
As a consequence, one has
\begin{equation}
\label{eq:enwithstar}
 E(v_n,A_{n,i})\ge E(\psi_\lambda,\S)+o(1), \qquad \text{for every}\quad i \in I_n
\end{equation}
and \eqref{eq:massnozero} and, since it is well-known that (see e.g. \cite{Adami15})
\begin{equation}
\label{eq-rettastar}
E(\psi_\lambda,\S)>\E_\R\big(\|\psi_\lambda\|_{L^2(\S)}^2\big)\,,
\end{equation}
one also has
\begin{equation}
	\label{eq:estimate2}
	E(v_n,A_{n,i})\ge \E_\R\big(\|\psi_\lambda\|_{L^2(\S)}^2\big)+o(1), \qquad \text{for every}\quad i \in I_n.
\end{equation}

In view of the previous estimates, we can see that in fact $\#I_n=1$, i.e. $A_n$ consists of a single connected component. Indeed, if, by contradiction $\# I_n\geq 2$, then by \eqref{eq:Efuori}, \eqref{eq:estimate} and \eqref{eq:estimate2} and the strict subadditivity of $\E_\R(\cdot)$
\[
\begin{split}
E(v_n,\G_n)=E(v_n,\G_n\setminus A_n)+E(v_n,A_n)&\,\ge\sum_{i\in I_n} E(v_n,A_{n,i})+o(1)\\
&\,\ge \sum_{i\in I_n}\E_{\R}\left(\Vert v_n\Vert^2_{L^2(A_{n,i})}\right) +o(1)>
\E_{\R}(1)+o(1).
\end{split}
\]
Since the continuity of $\E_\R(\cdot)$, \eqref{eq:massnozero} and the fact that $\sum_{i\in I_n}\Vert v_n\Vert^2_{L^2(A_{n,i})}\le 1$ imply that the previous inequality remains strict in the limit (possibly, up to subsequences) one gets a contradiction with \eqref{eq:encomp}.

However, if $A_n$ has a unique connected component, then $A_n=B_n$. This implies that the situation of case (ii) with $v_n(\mathbf{v}_n)\ge \delta>0$ cannot occur, since otherwise \eqref{eq:Efuori}, \eqref{eq:enwithstar} and \eqref{eq-rettastar}, together with the decreasing monotonicity of $\E_\R(\cdot)$ and the fact that \eqref{eq-convstargraph} entails here that $\|\phi_\lambda\|_{L^2(\R)}^2\leq1$, would contradict \eqref{eq:encomp} again.

Summing up, in any case we can find an open interval $B_n$ contained in a single edge $\widetilde{e}_n$ of $\G_n$ such that $E(v_n,\G_n\setminus \widetilde{e}_n)>0$ and \eqref{eq-convsoliton} holds. As a consequence, in order to avoid a contradiction with \eqref{eq-enGNnegative}, $\widetilde{e}_n$ is the unique edge where the energy of $v_n$ is negative. Hence, using \eqref{def:v} one obtains the claim.
\end{proof}

\begin{proof}[Proof of Theorem \ref{thm:noexlarge}]
    The result follows immediately by Remark \ref{rem:gsladder} and Proposition \ref{prop:bigmu}, simply observing that any edge of $\G$ is contained in at least one copy of $\L_1$.
\end{proof}


\section{Proof of Theorem \ref{thm:robust}}
\label{sec-robust}
In this section we discuss how the results in Theorems \ref{thm:noexsmall}--\ref{thm:noexlarge} change when a vertex of degree 1 is created in the macroscopically two-dimensional part of the graph, proving Theorem \ref{thm:robust}. 

Let then $\Gm$ be a defected grid obtained from $\G$ by removing a number of edges so as to create a vertex with degree 1 in the part of $\G$ that does not contain the horizontal copy of the ladder $\L_1$. A graph like this can be constructed by removing only two or three edges (see Figure \ref{fig:half+lad_1_mod}), even though Theorem \ref{thm:robust} applies also when more, but still finitely many, edges are removed. 

Let us first observe that the proof of Theorem \ref{thm:robust}(i) is identical to that of Theorem \ref{thm:noexsmall}. Indeed, with respect to the argument developed in Section \ref{sec-nonexistence}, the only difference is in one of the two modified quarter-grids, which will now contain the vertex of degree 1 of $\Gm$. However, it is readily seen that this does not affect the asymptotic behavior of small-mass ground states as proved in Proposition \ref{prop:2dQm}, and since the horizontal copy of the ladder graph $\L_1$ is identical both in $\G$ and in $\Gm$, the proof of Theorem \ref{thm:noexsmall} extends verbatim to prove Theorem \ref{thm:robust}(i). 

To conclude, we are then left to prove the existence result of Theorem \ref{thm:robust}(ii). As already observed at the beginning of Section \ref{sec-halfquart}, we know that any minimizing sequence $(u_n)_n\subset H_\mu^1(\Gm)$ admits a subsequence (not relabeled) such that $u_n\rightharpoonup u$ in $H^1(\Gm)$ and $u_n\to u$ in $L_{\text{\normalfont loc}}^\infty(\Gm)$ as $n\to+\infty$, and to prove existence of ground states with mass $\mu$ it is enough to show that $u\not\equiv0$ on $\Gm$.  

Assume by contradiction that $u\equiv0$ on $\Gm$. Since $\Gm$ differs from $\G$ by finitely many edges, there exists a compact set $K\subset \G$ containing all such edges. By the uniform convergence on compact sets, we then have $\displaystyle\|u_n\|_{L^\infty(K\cap\Gm)}\to0$ as $n\to+\infty$.  Furthermore, since $E(u_n,\Gm)=\E_{\Gm}(\mu)+o(1)$ and $\E_{\Gm}(\mu)$ is uniformly negative for $\mu$ bounded away from zero, it holds $\|u_n\|_{L^\infty(\Gm)}\not\to0$ as $n\to+\infty$. Hence, setting $v_n:=\kappa_n(u_n-\|u_n\|_{L^\infty(K\cap\Gm)})_+$, with $\kappa_n>0$ such that $\|v_n\|_{L^2(\Gm)}^2=\mu$ for every $n$, we obtain a new function $v_n\in H_\mu^1(\Gm)$ such that $\|v_n\|_{L^\infty(\Gm)}\not\to0$ as $n\to+\infty$, $\displaystyle \liminf_ nE(v_n,\Gm)\leq \lim_n E(u_n,\Gm)$, and $v_n\equiv0$ on $\Gm\cap K$. Since $\Gm$ and $\G$ differ only by some of the edges contained in $K$, we can then think of $v_n$ as a function supported on $\G$, and therefore
\[
\E_{\Gm}(\mu)=\lim_n E(u_n,\Gm)\geq\liminf_n E(v_n,\Gm)=\liminf_n E(v_n,\G)\geq\E_\G(\mu)\,.
\]
Combining with \eqref{eq:EG>ER}, this yields
\begin{equation}
	\label{eq:assurdo}
\E_{\Gm}(\mu)\geq\E_\R(\mu)\,.
\end{equation}
However, let $\varphi_\mu\in H_\mu^1(\R)$ be the unique positive radially symmetric decreasing function on $\R$ such that $E(\varphi_\mu,\R)=\E_\R(\mu)$. It is well-known (see e.g. \cite{Adami15}) that 
\[
\varphi_\mu(x):=\mu^\alpha\varphi_1(\mu^\beta x)\qquad\forall x\in\R\,,
\] 
where $\alpha, \beta$ are the numbers in \eqref{eq:ab} and $\displaystyle \varphi_1(x)=C_p\text{\normalfont sech}^{\alpha/\beta}(c_p\,x)$, for suitable constants $C_p,c_p>0$ depending only on $p$. This immediately shows that $\varphi_\mu(1)\to0$ and $\|\varphi_\mu\|_{L^2(\R\setminus[-1,1])}\to0$ as $\mu\to+\infty$. For every $\mu>0$, set then 
\[
v_\mu:=\kappa_\mu(\varphi_{2\mu}-\varphi_{2\mu}(1))_{|[0,1]}\,,
\]
where $\kappa_\mu>0$ is such that $\|v_\mu\|_{L^2(0,1)}^2=\mu$ for every $\mu$. Since $v_\mu$ is compactly supported on the interval $[0,1]$, we can actually think of $v_\mu$ as a function supported on the edge incident at the vertex of degree 1 of $\Gm$ (identifying the coordinate $x=0$ with this vertex), and then extending $v_\mu$ to zero on the rest of $\Gm$. In this way, we have $v_\mu\in H_\mu^1(\Gm)$, and direct computations show that
\[
\E_{\Gm}(\mu)\leq E(v_\mu,\Gm)=E(\varphi_{2\mu},[0,1])+o(1)=E(\varphi_{2\mu},\R^+)+o(1)=\E_{\R^+}(\mu)+o(1)
\]
as $\mu\to+\infty$, since it is well-known that the restriction to $\R^+$ of $\varphi_{2\mu}$ is a ground state with mass $\mu$ on the half-line. However, since for every $\mu>0$
\[
\E_{\R^+}(\mu)=\frac12\E_\R(2\mu)=-\frac12 C_p(2\mu)^{2\beta+1}=-C_p2^{2\beta}\mu^{2\beta+1}=2^{2\beta}\E_\R(\mu)<\E_\R(\mu),
\]
the previous inequality contradicts \eqref{eq:assurdo}, and the proof of Theorem \ref{thm:robust} is complete. 
 

\section*{Statements and Declarations}

\noindent\textbf{Conflict of interest} The authors declare that they have  no conflict of interest.

\medskip

\noindent\textbf{Data availability statement} Data sharing is not applicable to this article as it has no associated data. 

\medskip

\noindent\textbf{Acknowledgements.} The authors acknowledge that this study was carried out within the project E53D23005450006 ``Nonlinear dispersive equations in presence of singularities'' -- funded by European Union -- Next Generation EU within the PRIN 2022 program (D.D. 104 - 02/02/2022 Ministero dell'Universit\`a e della Ricerca). This manuscript reflects only the authors' views and opinions and the Ministry cannot be considered responsible for them.

W.B. and N.C. have also been supported by MUR grant \emph{Dipartimento di Eccellenza 2023-2027} of Dipartimento di Matematica at Politecnico di Milano.

\end{document}